\documentclass{amsart}
\usepackage{enumerate, amscd, amssymb, amsmath, mathtools}
\usepackage{hyperref}

\usepackage[english]{babel} 

\usepackage[utf8]{inputenc}

\allowdisplaybreaks
\usepackage{color}

\newcommand\II{\mathrm{I}\hskip-.3mm\mathrm{I}}

\newcommand\hide[1]{}
\newcommand\red[1]{\textcolor{red}{#1}}  

\usepackage{tikz}
\usepackage[normalem]{ulem}
\renewcommand\sout{\bgroup\markoverwith
{\textcolor{red}{\rule[0.7ex]{3pt}{1.4pt}}}\ULon}

\makeatletter
\newcommand{\doublewidetilde}[1]{{%
  \mathpalette\double@widetilde{#1}%
}}
\newcommand{\double@widetilde}[2]{%
  \sbox\z@{$\m@th#1\widetilde{#2}$}%
  \ht\z@=.9\ht\z@
  \widetilde{\box\z@}%
}

\newcommand{\dist}{\operatorname{dist}}
\newcommand\seq{\, = \,}
\newcommand\define{\mathrel{\ := \ }}
\newcommand\ede{{\define}}

\newcommand\manifb{M_0}

\newcommand{\CC}{\mathbb C}

\newcommand{\NN}{\mathbb N}

\newcommand{\RR}{\mathbb R}

\newcommand{\ZZ}{\mathbb Z}

\newcommand{\CI}{{\mathcal C}^{\infty}}
\newcommand{\CIc}{{\mathcal C}^{\infty}_{\text{c}}}

\newcommand\pa{\partial}

\newcommand\oN{\NN \cup \{\infty\}}
\newcommand\Mce{\widehat{M}}

\newcommand{\maA}{\mathcal A}
\newcommand{\maB}{\mathcal B}
\newcommand{\maC}{\mathcal C}

\newcommand{\maK}{\mathcal K}

\newcommand{\maU}{\mathcal U}
\newcommand{\maV}{\mathcal V}

\newcommand{\indexSet}{\mathfrak{I}}

\newcommand{\rinj}{\mathop{r_{\mathrm{inj}}}}

\newcommand\ball[2]{B_{#2}(#1)}

\newtheorem{theorem}{Theorem}[section]
\newtheorem{proposition}[theorem]{Proposition}
\newtheorem{corollary}[theorem]{Corollary}

\newtheorem{lemma}[theorem]{Lemma}

\theoremstyle{definition}
\newtheorem{assumption}[theorem]{Assumption}

\newtheorem{definition}[theorem]{Definition}
\newtheorem{notation}[theorem]{Notation}
\newtheorem{remark}[theorem]{Remark}
\newtheorem{example}[theorem]{Example}
\newtheorem{examples}[theorem]{Examples}

\author[B. Daniel]{Beno\^it Daniel} 
\address{Universit\'{e} de Lorraine, CNRS, IECL, F-54000 Nancy, France}
\email{benoit.daniel@univ-lorraine.fr}

\author[S. Labrunie]{Simon Labrunie} 
\address{Universit\'{e} de Lorraine, CNRS, IECL, F-54000 Nancy, France} 
\email{simon.labrunie@univ-lorraine.fr}

\author[V. Nistor]{Victor Nistor} \address{Universit\'{e} 
de Lorraine, CNRS, IECL, F-57000 Metz, France}
\email{victor.nistor@univ-lorraine.fr}

\thanks{V.N. has been partially supported by ANR grant 
OpART ANR-23-CE40-0016.}

\begin{document}

\title[Uniform estimates for a family]
{Uniform estimates for a family of Poisson problems: 
`rounding off the corners'}

\begin{abstract}
   We prove \emph{uniform solvability estimates} for certain 
   families of elliptic problems posed in a bounded family of 
   domains (for example, a sequence that converges to another 
   domain). We provide uniform estimates both in weighted and 
   in usual Sobolev spaces. When the limit domain is a \emph{polygon} 
   and the other domains are smooth, our results amount to 
   ``rounding off'' the corners of the limit domain. The technique of 
   proof is based on a suitable conformal modification of the metric, 
   which makes the union of the domains a manifold with boundary and 
   relative bounded geometry.
\end{abstract}

\maketitle

\tableofcontents

\section{Introduction}


We prove \emph{uniform solvability estimates} for certain families of
elliptic problems posed in a family of domains that is bounded in 
a certain sense. The two dimensional case and the higher dimensional 
case are somewhat different, so in this paper we include the results 
that are common to the two cases, as well as some of the results that 
are specific to the two dimensional case.

\subsection{Informal statement of the main result}
\label{ssec.statement}
While most of our results hold in all dimensions and for
suitably bounded families of domains, the most complete 
and easiest to state form of our main results is in two dimensions
for a convergent sequence of domains, 
so we first formulate informally our main results in this case. 
Thus let $$\Omega_{n} \subset \RR^{2}$$ 
be sequence of \emph{bounded, Lipschitz domains} in the plane, 
$n \in \NN := \{1, 2, \ldots\}$. We assume 
that this sequence ``converges'' in a certain sense as $n \to \infty$ 
to a domain $\Omega_\infty$ and we study to what extent we can 
obtain \emph{uniform} estimates for the solutions of the 
Poisson problems
\begin{equation}\label{eq.Poisson}
   \begin{cases}
      \ \Delta u_{n} \seq  f_{n} & \text{ in }\ \Omega_{n} \\ 
      \ u_{n} \seq 0 & \text{ on }\,  \partial  \Omega_{n}\,, \ \ 
      n \in \oN \,.
   \end{cases}
\end{equation} 

Let us assume that the limit 
domain $\Omega_{\infty} \subset \RR^{2}$ is a polygonal domain. 
We prove that, under suitable bounded geometry assumptions on the 
domains $\Omega_{n}$, $n \in \oN$, there exists $\delta_{U} > 0$ 
such that, for any $\delta < \delta_{U}$, there exists $C_{\delta} > 0$ 
such that, for any $n \in \oN$ and any $f_{n} \in L^2(\Omega_{n})$, 
the unique solution $u_{n} \in H^1(\Omega_{n})$ to the Poisson 
problem \eqref{eq.Poisson} satisfies the estimate
\begin{equation}\label{eq.needed.estimate}
   \|u_{n}\|_{H^{1 + \delta}(\Omega_{n})} 
   \le C_{\delta} \|f_{n}\|_{L^2(\Omega_{n})}\,.
\end{equation}
(See Theorem \ref{theorem.main}.) Our results remain true if the family 
$\Omega_{n}$, $n \in \NN$, is only bounded, in a certain sense (thus, not
necessarily convergent). Also, as mentioned above, most of our results 
are for domains $\Omega_{n}$ of $\RR^{d}$, $d \ge 1$.

\subsection{Weighted spaces and the setting}
\label{ssec.setting}
We now outline the setting, the assumptions, and the method 
of proof for our main result, Theorem \ref{theorem.main}, as 
well as the statements of some intermediate results.

\subsubsection{Weighted spaces} 
Most of our intermediate results are usually formulated in the  
Babu\v{s}ka-Kondratiev spaces and are interesing on their own.

\begin{definition}\label{def.BKspaces} 
   Let $G \subset \RR^{d}$ be an open subset and 
   $f : G \to (0, \infty)$ be a measurable 
   function. Then the 
   \emph{Babu\v{s}ka-Kondratiev space associated to $m \in \ZZ_+$,
   $a \in \RR$, the domain $G$, and the weight $f$} is the space 
   \begin{equation*}
      \maK_{a}^m (G; f) \ede \{\, u : G \to \CC \mid 
      f^{|\alpha| - a} \pa^\alpha u \in L^2(G; dx^2)\,, 
      \ |\alpha| \le m\, \}\,.
   \end{equation*}
\end{definition}

The Babu\v{s}ka-Kondratiev spaces are also called \emph{weighted
Sobolev spaces.} The Babu\v{s}ka-Kondratiev spaces are 
(essentially) the usual Sobolev spaces associated to some metrics
conformally equivalent to the Euclidean one. If $G$ is a 
manifold with boundary and relative bounded geometry 
in the metric $f^{-2}dx^2$ (see Section 
\ref{sec.bg}), then $\maK_{d/2}^m (G; f) = H^{m}(G; f^{-2}dx^2)$,
the usual Sobolev space on $G$ defined with respect to the 
metric $f^{-2}dx^2$, provided that 
$df/f \in W^{\infty, \infty}(G; f^{-2}dx^2)$ \cite{AGN_CR}
(a particular case of this result is proved in our paper,
for completeness). 
This fact will be important in our proofs.

\subsubsection{Summary of main notations and assumptions}
We next introduce in more detail our setting. We also formulate 
some of the assumptions needed \emph{for our main result.} 
All our domains will be \emph{open, bounded sets.} 
We shall let $\dist_{M}(x, y)$ denote the geodesic distance 
between two points $x, y$ of a Riemannian manifold $M$.

\begin{assumption}\label{assumpt.intro1}
   We are given an index set $\indexSet$, $N \in \NN$,
   and $1 \ge  R > 0$ such that the following will
   be valid throughout the paper:
   \begin{enumerate}[(i)]
      \item \label{item.intro.as12}
      For each $n \in \indexSet$, 
      $\maV_{n} = \{p_{1n}, p_{2n}, \ldots, p_{Nn}\} 
      \subset \RR^{d} \smallsetminus \Omega_{n}$ is a subset with 
      $N$ elements.
      
      \item\label{item.intro.as3} 
      For all $n \in \indexSet$ and all $1 \le i < j \le N$, we have
      $|p_{in} - p_{jn}| \ge R$.
      
      \item\label{item.intro.as4}  
      $\eta : [0,\infty) \to [0, R/6]$ 
      a fixed, smooth, non-decreasing function such that
      \begin{itemize}
         \item $\eta(t) = t$, for $t \le R/8$ and  
         \item $\eta(t) = R/6$, for $t \ge R/5$. 
      \end{itemize}
       
      \item $\hat g_{n} := r_{n}^{-2}dx^2$, where $r_{n}(x) \ede 
      \eta(|x - \maV_{n}|) \ede \eta( \dist_{\RR^{d}}(x, \maV_{n})).$
   \end{enumerate}
\end{assumption}

The function $r_{n}$ is thus a smooth function ``equivalent'' to 
$|x - \maV_{n}| := \dist_{\RR^{d}}(x, \maV_{n})$, the distance function 
to $\maV_{n}$ in $\RR^{n}$. However, in view of \eqref{item.intro.as3} 
above, $r_{n}$ is \emph{smooth outside} $\maV_{n}$ (unlike the function 
$|x - \maV_{n}|$). For simplicity, we shall assume also that $\indexSet$ is 
countable or finite (i.e. at most countable). 

Recall \cite{AGN1} that a manifold with \emph{finite width} is a manifold with 
boundary and relative bounded geometry such that the distance to the boundary 
is bounded globally (Definition \ref{def.finite.width}). These concepts are 
recalled in detail in Section \ref{sec.bg}. We next introduce the most significant 
assumption on the sets $\Omega_{n}$ and $\maV_{n}$, $n \in \indexSet$.

\begin{assumption}\label{assumpt.fw}
   For each $n \in \indexSet$, $\Omega_{n} \subset \RR^{d}$ 
   is a bounded open subset (a \emph{domain}) such 
   that $\Omega_{n} \cap \maV_{n} = \emptyset$
   and $\cup_{n \in \indexSet} (\Omega_{n} \cup \maV_{n})$
   is a bounded subset of $\RR^{d}$. Let 
   \begin{equation*}
      U \ede \cup_{n \in \indexSet} \{n\} \times 
   \Omega_{n} \subset \indexSet \times \RR^{d},
   \end{equation*} 
   with the metric 
   $\hat g$ equal to $\hat g_{n} := r_{n}^{-2}dx^2$ on 
   $\{n\} \times \Omega_{n} \equiv \Omega_{n} \subset \RR^{d}$,
   as in Assumption \ref{assumpt.intro1} above. We assume then that 
   $(U, \hat g)$ is a manifold with finite width (Definition \ref{def.finite.width}).
\end{assumption}

In addition to the metric $\hat g$, we endow 
$\indexSet \times \RR^{d}$ with the flat metric $dx^{2}$. 
Let $r = r_{n}$ on $\{n\} \times \RR^{d}$, then $\hat g = r^{-2} dx^{2}$. 
Since the domains $\Omega_{n} \subset \RR^{d}$ are such that 
$\maV_{n} \cap \Omega_{n} = \emptyset$, we can define 
$\maK_{a}^m(U; r)$ analogously to Definition \ref{def.BKspaces} 
(see Remark \eqref{rem.BKspaces}) and we have 
\begin{equation}\label{eq.H=K}
   \maK_{d/2}^m(U; r) \seq H^{m}(U; \hat g) \simeq 
   \oplus_{n \in \indexSet} H^{m}(\Omega_{n}; \hat g_{n})
   \simeq \oplus_{n \in \indexSet} 
   \maK_{d/2}^m(\Omega_{n}; r_{n})\,,
\end{equation} 
where $H^{m}(U; \hat g)$ is the usual Sobolev 
spaces associated to the metric $\hat g$. The method of 
proof will be to reduce the desired isomorphisms to 
isomorphisms between the spaces $H^{m}(U; \hat g)$ for
various $m$.

\subsection{Statement of the main result}
We shall assume in this subsection that $d = 2$ and hence consider 
the \emph{Laplacian} $\Delta = \pa_1^2 + \pa_2^2$ on
various spaces. It will always be defined on functions vanishing at 
the boundary, that is, we consider the \emph{Dirichlet Laplacian} $\Delta_D$.
In addition to this ``$D$,'' we may decorate the Laplacian $\Delta$ with some 
additional indices that will specify its domain (as in the next theorem).

\begin{theorem}\label{thm.wellp1}
   We assume \ref{assumpt.intro1} and \ref{assumpt.fw}.
   In particular, $r = r_{n} := \eta( _{\RR^{d}}(x, \maV_{n}))$ and 
   $\hat g = r_{n}^{-2}dx^2$ on $\Omega_{n}
   \equiv \{n\} \times \Omega_{n}$ and 
   $U := \cup_{n \in \indexSet} \{n\} \times \Omega_{n}.$
   We also assume that $(U, \hat g)$ 
   has \emph{finite width} (Definition \ref{def.finite.width}). 
   Then, there exists 
   $\delta_{U} > 0$ such that, for every $|a| < \delta_{U}$ and every
   $m \in \ZZ_+ := \{0, 1, \ldots\}$, the Dirichlet Laplacian 
   $\Delta_D$ induces an isomorphism
   \begin{equation*}
      \Delta_{D, m, a} : \maK_{a+1}^{m+1}(U; r) 
      \cap \{u\vert_{\pa U} = 0\}
      \to \maK_{a-1}^{m-1}(U; r)\,.
   \end{equation*}
\end{theorem}

This is a consequence of the well-posedness of the Laplacian on 
manifolds with finite width \cite{AGN1}, where the concept of 
``finite width'' was introduced. This result and the definiton 
of manifolds with finite width is recalled in Section \ref{sec.bg}.
When a single domain is considered (i.e., $\indexSet$ is reduced to 
a single element, the result of the last theorem is a classical theorem
due to Kondratiev \cite{Kondratiev67, KMR}).
Recall that if the domain $G$ is bounded, then 
$\Delta_D : H^1_0(G) \to H^{-1}(G)$ is an isomorphism.
An interpolation argument and the direct sum decomposition 
in \eqref{eq.H=K} then give the following consequence.

\begin{corollary}\label{cor.unif1}
   Under the assumptions of Theorem \ref{thm.wellp1}, for any 
   $\delta < \delta_{U}$, there exists a constant $C_\delta> 0$
   such that Equation \eqref{eq.needed.estimate} holds, that is,
   for any $n \in \indexSet$ and any $f_{n} \in L^2(\Omega_{n})$,
   we have
   \begin{equation*}
      \|\Delta_D^{-1}f_{n}\|_{H^{1 + \delta}(\Omega_{n})} 
      \le C_{\delta} \|f_{n}\|_{L^2(\Omega_{n})}\,.
   \end{equation*}
\end{corollary}

In Section \ref{sec.Benoit} we provide conditions on the sequences 
$(\Omega_{n}, \maV_{n})$ such that the conditions of Theorem \ref{thm.wellp1}
and its corollary, Corollary \ref{cor.unif1}, are satisfied (most notably, 
such that $(U, \hat g)$ be a manifold with boundary 
and relative bounded geometry). We also provide a concrete, explicit 
construction of such a sequence for any given curvilinear polygonal
domain $\Omega_{\infty}$. 

We allow the case $N = 0$, that is, the case when the sets 
$\maV_{n}$ are empty. Then the functions $r_{n}$ will be constant and 
the Babu\v{s}ka-Kondratiev spaces we are useing become the usual 
Sobolev spaces. This is the case, for instance, when the limit 
domain $\Omega_{\infty}$ is actually smooth (no corners). Even 
in that case, our results are new. See 
\cite{Olaf21, HenrotBook, Olaf22} for other types of boundedness
results for bounded (or convergent) sequences of domains.

\subsection{Contents of the paper} In Section \ref{sec.bg}, we
recall some basic material on manifolds with boundary and 
relative bounded geometry and we include a few direct applications 
of the results in those papers. In Section \ref{sec.ue}, 
we first recall some properties of the Babu\v{s}ka-Kondratiev
spaces. We then assume $d = 2$ and prove various uniform 
estimates under the assumption that $U$ has finite width. First
we prove these estimates in the framework of the 
Babu\v{s}ka-Kondratiev spaces (including Theorem \ref{thm.wellp1}).
Then we use the estimates in these spaces to obtain the uniform 
estimates in the usual Sobolev spaces (including our main 
Theorem, \ref{theorem.main}). Finally, in Section 
\ref{sec.Benoit}, given a straight polygonal domain $\Omega_{\infty}$,
we construct some completely explicit smooth,
bounded domains $\Omega_{n}$ such that the 
resulting manifold with boundary $(U, \hat g) := 
\cup (\Omega_{n}, \hat g_{n})$ is of relative bounded geometry. 
We also check that these sequences satisfy all the conditions of Assumptions 
\ref{assumpt.intro1} and \ref{assumpt.fw} are satisfied, so our results apply 
to these explicit approximating sequences of domains.
Our work generalizes the results obtained 
in~\cite{CiarletKaddouri,KaddouriThesis} in 
the case of a very specific rounding procedure. Rounding 
off corners has been considered in another setting in 
\cite{Brendle1, Brendle2} (we thank Bernd Ammann for 
bringing these papers to our attention).

\section{Background material}
\label{sec.background}

Our results rely on estimates on manifolds with bounded geometry,
so we first recall some basic material on manifolds with boundary and
relative bounded geometry, mostly from \cite{AGN1}, 
to which we refer for more
details and references. The presentation, however, is closer 
to \cite{GN17}. We also include a few direct applications of the 
results in those papers. In this section, we work on manifolds of
arbitrary dimension.

\subsection{Manifolds with bounded geometry}
\label{sec.bg}
Throughout this paper, $(M, g)$ will be a 
\emph{smooth Riemannian manifold} 
(without boundary) of dimension $d$, unless specified otherwise.
where, we recall, $g$ is a smoothly varying inner product on 
the tangent spaces $T_pM$, $p \in M$. If $E \to M$ is a smooth vector 
bundle, we let $\CI(M; E)$ be the space of smooth sections of $E$.
A \emph{connection} on $E$ is a differential operator 
\begin{equation}
   \begin{gathered}
      \nabla^E : \CI(M; E) \to \CI(M; T^*M \otimes E) \quad \mbox{satisfying}\\
      \nabla^E(f \xi) \seq df \otimes \xi + f \nabla^E \xi\,.
   \end{gathered}
\end{equation}
We also write $\nabla_X \xi := \langle X, \nabla \xi \rangle \in \CI(M; E)$
if $X \in \CI(M; TM)$, where the pairing $\langle \, , \, \rangle$ 
is the contraction $\CI(M; TM) \otimes \CI(M; T^*M) \to \CI(M)$.
If $E$ is endowed with a metric $g^E$ and $\nabla^E$ satisfies 
\begin{equation*}
   g^E(\nabla_X \xi, \eta) + g^E(\xi, \nabla_X \eta) \seq X g^E(\xi, \eta)\,,
\end{equation*} 
then we say that $\nabla^E$ is \emph{metric preserving.}
We endow $TM$ with the Levi-Civita connection $\nabla^M \colon 
\CI(M; TM) \to \CI(M; T^*M \otimes TM)$, which is the unique metric 
preserving connection on $TM$ that is also ``torsion-free,'' in the 
sense that $\nabla_X^M Y - \nabla_Y^M X = [X, Y]$ for all (smooth) vector 
fields $X, Y \in \CI(M; TM)$.

We endow each tensor bundle $T^{*\otimes k}M$ with the induced 
(Levi-Civita) tensor product connection:
\begin{equation*}
   \nabla_X (\xi_1 \otimes \ldots \otimes \xi_{k})
   := \nabla_X (\xi_1) \otimes \ldots \otimes \xi_{k} 
   + \xi_1 \otimes \nabla_X (\xi_2) \otimes \ldots \otimes \xi_{k} + 
   \ldots + \xi_1 \otimes \ldots \otimes \nabla_X (\xi_{k})\,.
\end{equation*}
We also endow all other tensor bundles with the induced tensor 
product connections, generically denoted by $\nabla$. (Thus we do not
try to always distinguish between the connections on different vector 
bundles, except for the Levi-Civita connection.)

\begin{definition}\label{def_ttly_bdd_curv} 
We say that $M$ has \emph{totally bounded curvature} 
if its curvatures $R^M := (\nabla^M)^{2}$ and all its covariant 
derivatives $(\nabla^M)^k R^M$ are bounded.
\end{definition}

Let $I \subset \RR$ be an interval. A $C^1$-curve $\gamma : I \to M$
is a \emph{geodesic} if $\nabla_{\gamma'(t)}^M \gamma'(t) = 0$. It is 
locally distance minimizing and uniquely determined by any $\gamma'(t_0)$,
$t_0 \in I$. If $M = \RR^{d}$ with the usual metric, then a geodesic is just
a (part of a) straight line.
Let $\exp^M\colon \maU \subset TM \to M$ be the geodesic exponential map
associated to the metric (where $\maU$ is a suitable open neighborhood of 
the zero section), that is, $\exp^M(v) = \gamma_v(1)$, where $\gamma_v$ is 
the unique geodesic with $\gamma'(0) = v$. 
By $\exp_p^M \colon \maU \cap T_pM \to M$ we denote the 
restriction of the geodesic 
exponential map to $T_pM \cap \maU$.

We shall need the \emph{injectivity radius} $\rinj(M)$ of a Riemannian 
manifold $(M, g)$. To define it, for any metric space $(X, d)$, 
$p \in X$, and $r \in \RR$, we let
\begin{equation}\label{eq.def.ball.r}
   B_{r}^{X}(x) \ede \{ y \in X \mid d(x, y) < r\}\,,
\end{equation}
that is, the \emph{open ball of radius $r$ and center $x$ in $X$.}
(We shall usually assume that $r > 0$, to obtain a non-empty set.)
In particular, $B_r^{M}(p)$ will be the open 
ball of radius $r$ in $M$ centered at $p \in M$. We endow $T_pM$ 
with the induced metric, so $B_r^{T_pM}(0)$ is the set of tangent 
vectors at $M$ in $p$ of length $< r$. With this preparation, we 
let 
\begin{equation}\label{eq.def.injr}
\begin{gathered}
   \mathop{r_{\mathrm{inj}}^M}(p) \define \sup\{ r \mid 
   \exp_p^M \colon B_r^{T_pM}(0) 
   \to B_r^{M}(p) \text{ is a diffeomorphism} \} \ \ \mbox{ and}\\
   \rinj(M) \define \inf_{p \in M} \mathop{r_{\mathrm{inj}}^M}(p).
\end{gathered}
\end{equation}

The following concept will play a crucial role in this paper.

\begin{definition}\label{def_bdd_geo_{n}o_bdy}
A boundaryless Riemannian smooth manifold $(M, g)$ is said to have
\emph{bounded geometry} if $M$ has totally bounded curvature
(Definition \ref{def_ttly_bdd_curv})
and $\rinj(M)>0$.
\end{definition}

There exist many equivalent definitions of manifolds with bounded 
geometry in the literature, see also \cite{AmannFunctSp, ElderingCR, 
GrosseSchneider, KordyukovLp1, KordyukovLp2, ShubinAsterisque}.
We now include some examples of manifolds with bounded geometry, 
some of which will be needed in what follows.

\begin{examples}\label{ex.bdg}
   The following are manifolds with bounded geometry:
   \begin{enumerate}[(i)]
      \item \label{item.ex.bdg1}
      A closed manifold (i.e. a smooth, compact manifold without 
      boundary).

      \item \label{item.ex.bdg2} $M = \RR^{d}$ with the standard 
      (Euclidean) metric 
      $$dx^2 := (dx_1)^2 + (dx_2)^2 + \ldots + (dx_d)^2\,.$$ 

      \item \label{item.ex.bdg3} 
      $D \times M$, where $M$ is a manifold with bounded geometry 
      and $D$ is discrete set. 
       (The analogous results holds true in the case of totally 
      bounded curvature: if $M$ has totally bounded courvature, 
      then so does $D \times M$.)

      \item \label{item.ex.bdg4} 
      $M_1 \times M_2$, where $M_1$ and $M_2$ have bounded geometry.
   \end{enumerate}
   Let us notice, however, that, if $M$ is a manifold with bounded 
   geometry and $U \subset M$ is an open subset, then $U$ has totally 
   bounded curvature, but may not have positive injectivity radius, so 
   $U$ may not have bounded geometry. 
\end{examples}

In relation to the phenomenon mentioned at the end of the last remark, 
we shall need the following concept (recall the notation of 
Equation \ref{eq.def.injr}):

\begin{definition}\label{def.rel.bg}
   Let $N$ be a Riemannian manifold with 
   totally bounded curvature. A subset $K \subset N$ is said 
   to be a subset of \emph{relative positive injectivity radius}
   if $$\inf_{x \in K} \mathop{r_{\mathrm{inj}}^{N}}(x) > 0.$$ 
\end{definition}

We have the following simple result (compare with Corollary 2.24 of \cite{AGN1}).

\begin{lemma}\label{lemma.gluing2}
   Let $(M, g)$ be a Riemannian manifold \emph{without boundary}
   and $W_i \subset M$, $i = 1, \ldots, k$, be open subsets.
   We assume that
   \begin{enumerate}[(i)]
      \item $(W_i, g)$ has totally bounded curvature for any $i = 1, \ldots, k$,
      \item there exist subsets $W_i' \subset W_i$ of relative positive injectivity 
      radius (Definition \ref{def.rel.bg}) such that $M = \bigcup_{i=1}^{k} W_i'.$  
   \end{enumerate}
   Then $(M, g)$ is a Riemannian manifold with bounded geometry. 
\end{lemma}

\begin{proof} The first condition ensures that $M$ has 
   totally bounded curvature. Let 
   \begin{equation*}
      r \ede \min_{1 \le i \le k} \big ( \inf_{x \in W_i'}\
      \mathop{r_{\mathrm{inj}}^{W_i}}(x) \big )\,.
   \end{equation*}
   Then $r > 0$, by the second assumption (see Definition \ref{def.rel.bg}). 
   Let $x \in M$ be arbitrary. Let 
   $i \in \{1, \ldots, k\}$ be such that $x \in W_i'$. Then
   the geodesic ball $B_r^M(x) \subset W_i$ of radius $r > 0$
   is a diffeomorphic image of the exponential map, again by the 
   second assumption. This shows 
   that the injectivity radius of $M$ is $\ge r > 0$.
\end{proof}

More to the point, we shall need that the manifold $\Mce$ that 
we will construct next has bounded geometry. To construct 
this manifold, we shall need the following. Let $| \, \cdot \, |$
denote the metric on $\RR^{d}$. \textbf{We assume from now 
on Assumption \ref{assumpt.intro1},} in particular, 
we assume that we are given a countable or finite index set $\indexSet$, 
$R > 0$, and sets $\maV_{n} = \{p_{1n}, p_{2n}, \ldots, p_{Nn}\}
\subset \RR^{d}$, $n \in \indexSet$, (each with $N$ elements) such 
that, for all $n \in \indexSet$ and all $1 \le i < j \le N$, we have
   $|p_{in} - p_{jn}| \ge R.$ 
We recall also that this assumption fixes a smooth, non-decreasing function 
$\eta$ satisfying 
\begin{equation*}
   \begin{cases}
      \ \eta : [0,\infty) \to [0, R/6]\\
      \ \eta(t) = t & \mbox{for } t \le R/8\\
      \ \eta(t) = R/6 & \mbox{for } t \ge R/5\,,
   \end{cases}
\end{equation*}
which then allows us to define
\begin{equation*} 
   r_{n} (x) \ede \eta(\min \{|x - p|, p \in \maV_{n}\}) \,.
\end{equation*} 
We have the following alternative formulas for $r_{n}$:
\begin{equation}\label{eq.def.rn}
   r_{n} (x) \ede \eta(\min \{|x - p|, p \in \maV_{n}\})
   \seq \eta( \dist_{\RR^{d}}(x, \maV_{n})) \seq \eta(|x - \maV_{n}|)\,,
\end{equation}
where $\dist_{\RR^{d}}(x, \maV_{n})$ is the distance from $x$ to the set 
$\maV_{n}$ in the Euclidean metric on $\RR^{d}$, as we did in 
the Introduction, Section \ref{ssec.setting}.

The conditions on $\eta$ in Assumption \ref{assumpt.intro1} (recalled 
above) are stronger than needed in this section, but 
they will be useful in the next section. By rescaling, we may assume that 
$R = 1$, which we shall do when convenient.

\begin{proposition}\label{prop.Mce.bg}
   Let $M_{n} := \RR^{d} \smallsetminus \maV_{n}$ be endowed with the 
   metric $\hat g_{n} := r_{n}^{-2}dx^2$ (see Assumption \ref{assumpt.intro1}
   or Equation \ref{eq.def.rn} and the discussion preceeding it).
   Let $\Mce := \cup_{n \in \indexSet} \{n\} \times M_{n}$ 
   be the \emph{disjoint} union of the manifolds
   $M_{n}$, $n \in \indexSet$. We endow $\Mce$ with the metric $\hat g$ such 
   that $\hat g = \hat g_{n}$ on $\{n\} \times M_{n} \equiv M_{n}$, as before. 
   Then $\Mce$ has bounded geometry.
\end{proposition}

\begin{proof}
   Let us rescale our metric so that $R = 1$ in Assumption \ref{assumpt.intro1}.
   We shall use Lemma \ref{lemma.gluing2}. We denote 
   $\maV_{n} = \{ p_{1n}, p_{2n}, \ldots,  p_{Nn} \},$ as usual.
   For each $n \in \indexSet$, we let 
   \begin{equation}
      \begin{cases}
         \ A_{n} \ede \{ x \in \RR^{d} \mid\ \exists\, i \in \{1, \ldots, N\}\
         \mbox{ such that }\ |x - p_{in}| < \frac 18\}\,,
         \\
         \ B_{n} \ede \{ x \in \RR^{d} \mid\ \exists\, i \in \{1, \ldots, N\}\ 
         \mbox{ such that }\ |x - p_{in}| \in (\frac 1{16}, \frac 12)\}\\
         \ C_{n} \ede \{ x \in \RR^{d} \mid\ \forall\, i \in \{1, \ldots, N\}\
         \mbox{ we have }\ |x - p_{in}| > 
         \frac 15\}\,.
      \end{cases}
   \end{equation}
   The choice of the parameters $\frac 18$ and $\frac 15$ comes from 
   the definition of $\eta$, whereas the other parameters are such 
   that $\frac 1{16} < \frac 18 < \frac 15 < \frac 12$.
   The last condition ensures, in particular, that  
   $A_{n} \cup B_{n} \cup C_{n} = \RR^{d}$, for each $n \in \indexSet$.
   We then let $W_A := \cup_{n \in \indexSet} \{n\} \times A_{n} 
   \subset M$ be the \emph{disjoint} union of the sets 
   $A_{n}$, $n \in \indexSet$. Similarly, we define 
   $W_B := \cup_{n \in \indexSet} \{n\} \times B_{n}$ and 
   $W_C := \cup_{n \in \indexSet} \{n\} \times C_{n}$. 
   We endow all these sets with the metric $\hat g := r^{-2}dx^{2}$.
   
   Let $\maA := B_{1/8}^{\RR^{d}} \smallsetminus \{0\}$ and 
   $\maB := B_{1/2}^{\RR^{d}} \smallsetminus \overline{B}_{1/16}^{\RR^{d}}$,
   both being endowed with the metric $\eta(|x|)^{-2} dx^2$.
   Because, for each $n \in \indexSet$, we have that the balls 
   $B_{1/2}^{\RR^{d}}(p_{in})$ are disjoint and because of how the 
   metrics $\hat g_{n} := r_{n}^{-2}dx^2$ were constructed, we have natural 
   isometries (given by translations)
   \begin{equation}\label{eq.WA}
      W_A \simeq \maA \times \indexSet \times \{1, 2, \ldots, N\}
   \end{equation} 
   and 
   \begin{equation}\label{eq.WB}
      W_B \simeq \maB \times \indexSet \times \{1, 2, \ldots, N\}\,.
   \end{equation} 
   Notice next    that, on $\maA$, we have 
   $\hat g = \eta(|x|)^{-2} dx^2 = |x|^{-2} dx^2$, by the definitions.
   Let $S^{d-1}$ denote the unit sphere in $\RR^{d}$.
   Hence the manifold $(\maA, \hat g)$ is isometric to $S^{d-1} \times (0, \infty)$,
   so it has totally bounded curvature, and hence so does $W_A$, by 
   Equation \eqref{eq.WA} and by Example \ref{ex.bdg}\eqref{item.ex.bdg3}.
   Moreover, the set $\maA' := \overline{B}_{1/12}^{\RR^{d}} \smallsetminus \{0\}$ 
   is a subset of $\maA$ of relative positive injectivity radius 
   (Definition \ref{def.rel.bg}).
   Let $W_A' \simeq \maA' \times \indexSet \times \{1, 2, \ldots, N\}$
   be the corresponding subset of $W_A$, which will hence analogously be 
   a subset of $W_A$ of relative positive injectivity radius. 
   
   Similarly, the manifold $\maB := B_{1/2}^{\RR^{d}} \smallsetminus 
   \overline{B}_{1/16}^{\RR^{d}}$ is a relatively compact annulus 
   conformally equivalent to the usual one (via the factor $r_{n}$
   on each component), so it also has totally bounded curvature. Hence 
   $W_B$ also has totally bounded curvature.
   Again, the set $\maB' := \overline{B}_{1/3}^{\RR^{d}}
   \smallsetminus B_{1/12}^{\RR^{d}}$ is a subset of $\maB$
   of relative positive injectivity radius. Analogously, let 
   $W_B' \simeq \maB' \times \indexSet \times 
   \{1, 2, \ldots, N\}$ be the corresponding subset of $W_B$ of relative 
   positive injectivity radius.

   Finally, each $C_{n}$ is an open subset of $\RR^{d}$ with the induced 
   Euclidean metric so $W_C$ has totally bounded curvature as well. 
   A direct verification shows that $$C_{n}' := 
   \{ x \in \RR^{d} \mid\ \forall\, i \in \{1, \ldots, N\}\
   \mbox{ we have }\ |x - p_{in}| \ge \frac 13\}$$ is a  
   subset of $C_{n}$ of relative positive injectivity radius, and hence that  
   that $W_C' := \cup_{n \in \indexSet} C_{n}'$ is a subset of $W_C$ 
   of relative positive injectivity radius (this is just the verification that the 
   ball of radius $\frac 13 - \frac 15$ around each point of $x \in W_C'$
   is contained in $W_C$). By construction, we have 
   \begin{equation*}
      W_A' \cup W_B' \cup W_C' \seq M\,.
   \end{equation*}
   Hence the sets $W_A$, $W_B$, and $W_C$ satisfy both conditions 
   of Lemma \ref{lemma.gluing2}. That lemma yields then that $M$ is of 
   bounded geometry.
\end{proof}

We also have the following corollary of the proof. 

\begin{corollary}\label{cor.admissible.rn}
   We use the notation of Proposition \ref{prop.Mce.bg}, in particular, 
   $r_{n}(x) := \eta(|x-\maV_{n}|)$ and $r = r_{n}$ on $M_{n} \subset \Mce$. 
   For each multi-index $\alpha = (\alpha_{1},
   \alpha_{2}, \ldots, \alpha_{d}) \in \ZZ_{+}^{d}$ with 
   $|\alpha| := \alpha_{1} + \alpha_{2} + \ldots + \alpha_{d}$, 
   and each $b \in \RR$, we have that 
   \begin{equation*}
      r^{-b} (r \pa)^{\alpha} r^{b}  \, , \
      r^{|\alpha|-b} \pa^{\alpha} r^{b}
      \in L^{\infty}(\Mce)
   \end{equation*}
   and that they depend smoothly on $b$,
   where $(r\pa)^{\alpha} u := (r\pa_{1})^{\alpha_{1}}
   (r\pa_{2})^{\alpha_{2}} \ldots (r\pa_{d})^{\alpha_{d}}u$.
\end{corollary}

\begin{proof}
   Since all the functions $r_{n}$ are constructed out of the single function 
   $\eta(|x|)$ around each ball $B_{R/5}^{M}(p_{jn})$, are constant outside 
   these balls, and these balls are disjoint, it is enough to prove the result 
   for $r = \eta(|x|)$ (the function in a single ball).
   In other words, we assume that $\indexSet$ is reduced to a single element
   $n$ and $\maV_{n}$, in turn, is also reduced to a single element. 
   The proof in this case is elementary (and essentially known). We include the main 
   steps of the calculation for the benefit of the reader. Let 
   $\rho(x) := |x|$ and $r(x) = \eta(\rho(x))$, as before. We consider 
   generalized polar coordinates $(\rho, x') \in [0, \infty)\times S^{d-1}$,
   where $S^{d-1}$ denotes the unit sphere in $\RR^{d}$, as usual. We first 
   notice that $\rho := (x_{1}^{2} + x_{2}^{2} + \ldots + x_{d}^{2})^{1/2}$, 
   $\pa_{j} \rho = x_{j}/\rho =: x_{j}'$, and 
   \begin{equation*}
      \rho \pa_{j} x_{i}' \seq \delta_{ij} - x_{i}' x_{j}'
   \end{equation*}
   are all smooth functions on $[0, \infty)\times S^{d-1}$, and hence 
   $\rho \pa_{j}$ maps $\CI([0, \infty)\times S^{d-1})$ to itself. Hence 
   also 
   \begin{equation}\label{eq.aux.selfmap}
      r \pa_{j} : \CI([0, \infty)\times S^{d-1}) \to 
      \CI([0, \infty)\times S^{d-1})\,,
   \end{equation}
   because the only problem can occur near $r = 0$ and $r \pa_{j} = \rho \pa_{j}$ 
   there. Since $\eta(t) = t$ for $t$ small and $\eta(t) =$ constant for 
   $t$ large, our previous calculation gives that 
   \begin{equation*}
      r^{-b}(r\pa_{j}) r^{b} \seq b \pa_{j}r \ede b \pa_{j} \eta(\rho)
      \seq b \eta'(\rho) \pa_{j}\rho 
   \end{equation*}
   defines a \emph{smooth function} on $[0, \infty) \times S^{d-1}$ that 
   vanishes at infinity. We then obtain 
   by induction on $|\alpha|$ that $r^{-b} (r \pa)^{\alpha} r^{b}$ is also a 
   \emph{smooth function} on $[0, \infty) \times S^{d-1}$ that 
   vanishes at infinity as follows. Indeed, let $\phi_{\alpha} 
   := r^{-b} (r \pa)^{\alpha} r^{b}$. Then 
   \begin{equation*}
      r^{-b} r \pa_{j} (r \pa)^{\alpha} r^{b}  \seq r^{-b} r \pa_{j} 
      [r^{b} \phi_{\alpha}] 
      \seq \phi_{\alpha} r^{-b} r \pa_{j} r^{b} + r \pa_{j} 
      \phi_{\alpha} \,,
   \end{equation*}
   which is smooth in generalized spherical coordinates (that is, as a function 
   on $[0, \infty) \times S^{d-1}$) in view of the induction hypothesis
   and of Equation \eqref{eq.aux.selfmap}. Since 
   $r^{-b} (r \pa)^{\alpha} r^{b}$
   is a smooth function on $[0, \infty) \times S^{d-1}$ that vanishes outside 
   a compact set, it is bounded (i.e. in $L^{\infty}(\Mce)$, as claimed). 
   If we regard this function as a function of 
   $b$ also, the above calculations show that it is a smooth function on 
   $\RR \times [0, \infty) \times S^{d-1}$, which yields the smooth 
   dependence on $b \in \RR$. This proves the first half of our result. 
   
   Let us prove next the second half of our result. It is enough to prove
   that $r^{|\alpha|-b} \pa^{\alpha} r^{b}$ extends to a smooth function on 
   $\RR \times [0, \infty) \times S^{d-1}$ that vanishes for $r$ large. To 
   this end, we shall use the facts already proved and induction on 
   $|\alpha|$. For $|\alpha| = 1$, the result has already been proved 
   (when $|\alpha| = 1$, the two halves of our result are, actually, the 
   same). The induction step follows from
   \begin{equation}\label{eq.ca.aici}
      r^{|\alpha| + 1 -b} \pa_{j} \pa^{\alpha} r^{b} \seq r \pa_{j} 
      \big ( r^{|\alpha|-b} \pa^{\alpha} r^{b} \big ) - 
      (|\alpha|-b) (\pa_{j} r) r^{|\alpha| - b} \pa^{\alpha} r^{b} \,,
   \end{equation}
   Equation \eqref{eq.aux.selfmap} and the fact that $\pa_{j} r$ is smooth 
   on $[0, \infty) \times S^{d-1}$, a fact already proved. This 
   completes the proof.
\end{proof}

The above corollary is well-known for one function $r_{n}$,
the point of our result being that we obtain \emph{uniform} 
estimates in $n \in \indexSet$; in the terminology of \cite{AGN_CR},
it means that $r^{b}$ is an admissible weight. We let 
\begin{equation}\label{eq.def.Winfinf}
   W^{\infty, \infty}(M; E) \ede \{u \mid \nabla^{k} 
   u \in L^{\infty}(M; E), \in \ZZ_{+}\} \,,
\end{equation}
and we notice that this space does depend on the metrics on 
$E$ and on $M$ via the covariant differential and the metric 
on $TM$ (unlike $L^\infty(M)$). We record the following simple 
technical lemma for further use.

\begin{lemma}\label{lemma.diff.ops}
   We can write $(r\pa)^{\alpha}$ as a linear 
   combination of operators of the form 
   $a_{\beta} r^{|\beta|} \pa^{\beta}$ with 
   $a_{\beta} \in W^{\infty, \infty}(\Mce)$ and $\beta \le \alpha$. 
   Similarly, $r^{|\alpha|}\pa^{\alpha}$ can be written as a linear 
   combination of operators of the form $a_{\beta} (r \pa)^{\beta}$
   with $a_{\beta} \in W^{\infty, \infty}(\Mce)$ and $\beta \le \alpha$.
\end{lemma}

\begin{proof} 
   This is a direct (albeit long) calculation. In one direction
   it follows by iterating Equation \eqref{eq.ca.aici}. In the 
   other direction, it is a very similar calculation.
\end{proof}

\subsection{Manifolds with boundary and relative bounded geometry}
\label{ssec.bg2}
Let now $\manifb$ be a Riemannian manifold \emph{with boundary,} 
then $\rinj(\manifb)=0$, so a manifold with non-empty
boundary will never have bounded geometry in the sense of the above
definition. The way around this conundrum was found by Schick 
\cite{Schick2001}, who has defined the concept of ``manifold with 
boundary and bounded geometry,'' (we shall call these manifolds 
``manifolds with boundary and \emph{relative} bounded geometry,'' 
to avoid confusions). We recall the equivalent definition of 
manifolds with boundary and relative bounded geometry in \cite{AGN1}. 
The main point of that definition is to assume that the boundary 
$\pa \manifb$ of $\manifb$ is a suitable submanifold of a 
(boundaryless) manifold $M$ with bounded geometry. 

Let hence $M$ be a (boundaryless) manifold with bounded geometry 
and let us consider a hypersurface $H \subset M$, that is, a 
submanifold $H$ of $M$ of codimension $\dim (M) - \dim (H) = 1$. 
We assume that $H$ carries a globally defined unit normal vector 
field $\nu$. We let 
\begin{equation}\label{eq.exp.perp}
   \exp^\perp(x,t) \define \exp_x^M(t\nu_x)
\end{equation} 
be the exponential in the direction of the chosen unit normal vector. 
We shall need the \emph{second fundamental form} $\II^{H}$ of $H$ in 
$M$, which, we recall, is defined by $\II^H(X,Y)\nu \define \nabla^{M}_XY - \nabla^{H}_XY$, 
\cite[page 149]{Petersen}, where $\nabla^{Z}$ is the Levi-Civita connection 
of a Riemannian manifold $Z$. Equivalently, since $g(\nu,\nabla^{H}_XY)=0$, we have 
$\II^H(X,Y) \define g(\nu, \nabla^{M}_XY)$.

\begin{definition}\label{def.hSurfaceBG} 
Let $(M, g)$ be a Riemannian manifold of bounded geometry and
$H \subset M$ be a hypersurface with unit normal vector field $\nu$ on
$H$. We say that $H$ is a \emph{bounded geometry hypersurface in $M$}
if it satisfies the following conditions:
\begin{enumerate}[(i)]
\item $H$ is a closed subset of $M$;
\item all covariant derivatives $(\nabla^H)^k \II^H$, $k \ge 0$, 
are bounded;
\item $\exp^\perp\colon H\times (-\delta,\delta)\to M$ is a
  diffeomorphism onto its image for some $\delta > 0$ (see 
  Equation \ref{eq.exp.perp}).
\end{enumerate}
We shall denote by $r_{H}$ the largest value of $\delta$ 
satisfying this definition.
\end{definition}

If $H$ is as in the above definition, then it has bounded geometry 
\cite{AGN1}. The assumption that $H$ has a globally defined unit normal 
vector field is just for convenience, but will always be satisfied in 
this paper. It is convenient for technical reasons to allow $H$
in the above definition to have a non-empty boundary.

\begin{example}\label{ex.bg.H} 
   Let $H \subset M$ be a 
   \emph{compact} hypersurface (it may have a non-empty boundary).
   Then $H$ is a bounded geometry hypersurface in $M$.
\end{example}

We are ready now to recall the definition of a central concept 
in this paper.

\begin{definition}\label{def_bdd_geo}  
   We shall say that $\manifb$  \emph{is a manifold with boundary and 
   relative bounded geometry} if $\manifb$ is isometrically contained 
   in a (boundaryless) Riemannian manifold $M$ with
   bounded geometry such that $\partial \manifb$ 
   is a bounded geometry hypersurface in $M$.
\end{definition}

Note that we use the term \emph{``manifold with boundary and 
\underline{relative} bounded geometry,''} which we think is more 
precise than the term 
\emph{``manifold with boundary and bounded geometry''} used in
\cite{AGN1, Schick2001} and in other papers; in particular, it 
may help avoid some confusions.

We will also need the following results concerning hypersurfaces.
Let $r_{H}$ be as in Definition \ref{def.hSurfaceBG}.
If $H$ is a compact hypersurface of $M$, we let $\omega_H$ be the largest 
$\delta>0$ (possibly $\infty$) such that 
$H\cap\exp^\perp(H\times(-\delta,0)\cup(0,\delta))=\emptyset$.

\begin{lemma} \label{lemma.normalinjectivity}
We use the notation recalled in the above paragraph. Let $A\geq0$. 
Then there exists a constant $C(A)>0$ such that the 
following statement holds.

Let $(M, g)$ be a complete Riemannian manifold whose sectional curvature 
is bounded from above by $A$. Let $H$ be a compact hypersurface of $M$ 
without boundary such that $||\II_H||_{L^\infty}\leq A$. Then 
$r_H\geq\min\left(C(A),\frac12\omega_H\right)$ (Definition 
\ref{def.hSurfaceBG}). 
\end{lemma}

\begin{proof}
   This follows from well-known arguments; see for instance 
   \cite{Boileau}, page 48, and \cite{Funar}, page 389. We review 
   this argument for the benefit of the reader. 
   
   Let $f_H$ be the focal radius of $H$, i.e.,  the largest largest 
   $\delta>0$ (possibly $\infty$) such that 
   $\exp^\perp:H\times(-\delta,\delta)\to M$ 
   is an immersion. It follows from the definition of $r_{H}$
   (Definition \ref{def.hSurfaceBG}) and that of $f_{H}$ that 
   $0<r_H\leq f_H$.
   By the proof of Lemma 2.3 in \cite{FunarGrimaldi}, there exists a 
   constant $C(A)>0$, depending only on $A$, such that 
   $f_H\geq C(A)$ (see also \cite{Boileau}, page 48, and \cite{Warner}, 
   Corollary 4.2). Thus, if $r_H = f_H$, the proof is complete.
      
   On the other hand, if $r_H<f_H$, then there exists a geodesic of length $2r_H$ that orthogonally intersects $H$ at its endpoints $a_1$ and $a_2$
   (see \cite{Hermann}, Theorem 4.2, and \cite{Kodani}, Lemma 6.3). In this case we have $\exp^\perp(a_1,2r_H)=a_2\in H$ or $\exp^\perp(a_1,-2r_H)=a_2\in H$, and so $2r_H\geq\omega_H$. This concludes the proof.
\end{proof}

\begin{lemma}\label{lemma.conformal2}
   Let $(M, g)$ be a Riemannian manifold and $\sigma$ a smooth function on $M$. 
   Let $H$ be hypersurface of $M$ with unit normal vector field $\nu$. Let $\II^H$ 
   be its second fundamental form with respect to $\nu$. Let $\II_\sigma^H$ be the 
   second fundamental form of $H$ in $M$ endowed with the metric $\sigma^2g$ with 
   respect to $\frac1\sigma\nu$. Then $\II^H$ and $\II_\sigma^H$
   are related by   
     $$\II^H_\sigma(X,Y)=\sigma\II^H(X,Y)-g\left(\nu,\nabla^{M}\sigma\right)g(X,Y)$$
   for all vector fields $X$ and $Y$ tangent to $H$, where $\nabla^{M}$ is the 
   Levi-Civit\`a connection of the metric $g$ on $M$.
\end{lemma}

\begin{proof}
   This formula is well-known. We include a proof to make our paper
   more self-contained. Let $\nabla^\sigma$ be the Levi-Civit\`a connection of the 
   metric $\sigma^2g$. By Koszul's formula \cite[page 25]{Petersen} we have, for 
   all vector fields $X$, $Y$, $Z$ on $M$,
   \begin{eqnarray*}
   2g(\nabla_XY,Z) & = & X(g(Y,Z))+Y(g(X,Z))-Z(g(X,Y)) \\
   & & -g([Y,X],Z)-g([X,Z],Y)-g([Y,Z],X)
   \end{eqnarray*}
   and
   \begin{align*}
      2\sigma^2g(\nabla^\sigma_XY,Z) & \seq  
      X(\sigma^2g(Y,Z))+Y(\sigma^2g(X,Z))-
      Z(\sigma^2g(X,Y))\\
      & \ \ 
      -\sigma^2g([Y,X],Z)-\sigma^2g([X,Z],Y)-
      \sigma^2g([Y,Z],X) \\
      & \seq 2 \big (\sigma^2g(\nabla_XY,Z)
      +\sigma g(\nabla\sigma,X)g(Y,Z)
      +\sigma g(\nabla\sigma,Y)(X,Z)-
      \sigma g(\nabla\sigma,Z)(X,Y) \big),
   \end{align*}
   so
   $$\nabla^\sigma_XY=\nabla_XY
   +\frac1\sigma\left(g(\nabla\sigma,X)Y + 
   g(\nabla\sigma,Y)X-g(X,Y)\nabla\sigma\right).$$
    
   On the other hand, we have, for $X$, $Y$ tangent to $H$,
    $$\II^H(X,Y)=g(\nu,\nabla_XY),$$
    $$\II^H_\sigma(X,Y)=\sigma^2g\left(\frac1\sigma\nu,\nabla^\sigma_XY\right),$$
    and so
     $$\II^H_\sigma(X,Y)=\sigma g\left(\nu,\nabla_XY
   -\frac1\sigma g(X,Y)\nabla\sigma\right)
   =\sigma\II^H(X,Y)-g\left(\nu,\nabla\sigma\right)g(X,Y)\,,$$
   as claimed.
\end{proof}

\subsection{Sobolev spaces and well-posedness of the Poisson problem}
\label{ssec.cov}
We now recall the definition of $L^2$-type Sobolev spaces on manifolds
with boundary and relative bounded geometry and discuss the 
Poincar\'e inequality and its consequences. Thus \emph{for the rest of the 
paper, $\manifb$ will denote a manifold with boundary and relative bounded 
geometry and $M$ will be a (boundaryless) manifold with 
bounded geometry containing $\manifb$, as in Definition \ref{def_bdd_geo}.} 
We agree that $\pa \manifb \subset \manifb$ (so $\manifb$ is \emph{not} an 
open subset of $M$).

Let $s \in \ZZ_{+}$. We can define the space $H^s(M; g)$ as the domain of 
$(1 - \Delta_{g})^{s/2}$, where $\Delta_s := -d^*d$ is the 
\emph{non-positive} Laplacian on $M$. Equivalently, $H^{s}(M; g) =
\{u \mid \nabla^{j} u \in L^{2}(M; g), \forall \, j \le k\}$.
We let $H^s(M; g) := H^{-s}(M; g)^*$ if $-s \in \NN$.
Let $G \subset M$ be an open subset or $G = \manifb$,
then we define $H^s(G; g)$ to be the space of restrictions distributions 
in $H^s(M; g)$ to $G$, for $s \in \ZZ$. For $s \in \RR\smallsetminus \ZZ$,
we define $H^s(G; g)$ by interpolation between nearby integers.
We drop the metric $g$ from the notation when it is clear from the context.
Let $\nu$ be the inner unit normal vector field of $\partial \manifb$.
We have the following result \cite{GrosseSchneider, TriebelBG}.

\begin{theorem}[Trace theorem]\label{thm.trace}
   Let $\manifb$ be a manifold with boundary and relative bounded geometry. 
   Then, for every $s > 1/2$, the restriction $\text{res}\colon
   \CIc(\manifb) \to \CIc(\pa \manifb)$ extends by continuity to a
   surjective map
   \begin{equation*}
      \operatorname{res}\colon  H^s(\manifb; g) \, \to \, 
      H^{s-\frac{1}{2}}(\partial \manifb; g).
   \end{equation*}
\end{theorem}

Let $\pa_{\nu}$ be the normal derivative at the boundary (with respect to 
either outer or inner normal vectors). We then let $H_0^k(\manifb; g) :=
\cap_{j=0}^{k-1} \ker (\operatorname{res}\circ \pa_{\nu}^j)$ denote 
the joint kernel of the restrictions maps 
$\operatorname{res}\circ \pa_{\nu}^j$, for $0\leq j\leq k-1$. It is known 
that $\CIc(\manifb \smallsetminus \pa \manifb)$
is dense in $H_0^{k}(\manifb; g)$ and that $H^{-k}(\manifb; g)$ 
identifies with $H^k_0(\manifb; g)^*$, $k \in \NN$. See 
\cite{GrosseSchneider, TriebelBG}.

We next recall the (uniform) Poincar\'e inequality of \cite{AGN1, Sakurai}
and (as a consequence), the (uniform) well-posendess of the Dirichlet 
Laplacian on manifolds with finite width from 
\cite{AGN1}. To that end, we need first to recall the notion of manifold
with finite width.

\begin{definition}\label{def.finite.width}
  Let $(\manifb, g)$ be a Riemannian manifold with boundary 
  $\partial \manifb$.
  We say that $\manifb$ has \emph{finite width} if:
  \begin{enumerate}[(i)]
  \item $(\manifb, g)$ is a manifold with boundary and 
  relative bounded geometry and
  \item\label{cond.fw} The function $\manifb \ni x \to 
  \dist_{\manifb}(x, \pa \manifb)$ is bounded on $\manifb$.
  \end{enumerate}
\end{definition}

Condition \eqref{cond.fw} of the last definition is 
equivalent to the condition 
\begin{center}
``\,$\exists \, R >0$ such that 
$\manifb \subset \{x \in M \mid\, \exists y \in \pa \manifb,\, 
\dist_{\manifb}(x, y) < R\,\}$.''
\end{center}

We can now recall the following Poincar\'e inequality 
from \cite{AGN1, Sakurai}.

\begin{theorem}\label{thm_Poin_intro}
   Let $(\manifb, g)$ be a Riemannian manifold with \emph{finite width}. 
   Then there exists $0 < C_{\manifb} < \infty$ such that, for all 
   $f \in \CIc(\manifb)$ (thus $f = 0$ on the boundary of $\manifb$),
   \begin{equation*} 
      \Vert f\Vert _{L^{2}(\manifb)} \leq C_{\manifb} 
      \Vert d f\Vert _{L^{2}(\manifb)}\,.
   \end{equation*} 
\end{theorem}

A regularity argument then yields the following result \cite{AGN1}.

\begin{theorem} \label{thm.well.posedness} 
   Let $\manifb$ be a smooth Riemannian manifold with smooth boundary 
   $\partial \manifb$ and finite width (Definition 
   \ref{def.finite.width}). Then $\Delta$ induces isomorphisms
   \begin{equation*}
      \Delta_D \seq \Delta \colon H^{m+1}(\manifb) \cap H_0^1(\manifb) 
      \to H^{m-1}(\manifb) \,, \quad m \in \ZZ_{+} := \{0, 1, \ldots \}\,.
   \end{equation*}
\end{theorem}

We shall need some of the following results.

\begin{examples}\label{ex.bdg_bdry}
   We have the following analogues of Example \ref{ex.bdg}:
   \begin{enumerate}[(i)]
      \item \label{item.ex.bbdg1}
      A bounded domain $G$ with smooth boundary in a 
      Riemannian manifold has relative bounded geometry.

      \item \label{item.ex.bbdg3} If $\manifb$ is a manifold with 
      boundary and relative bounded geometry 
      and $D$ is any discret space, then 
      $D \times \manifb$ is again a manifold with boundary and 
      relative bounded geometry. 

      \item \label{item.ex.bbdg4} If $M$ is a manifold with bounded geometry and 
      $\manifb$ is a manifold \emph{with boundary and relative} bounded geometry
      then $\manifb \times M$ is also a manifold \emph{with boundary and relative}  
      bounded geometry.
   \end{enumerate}
   On the other hand, let $1 > r_{n} \searrow 0$, $n \in \NN$. 
   Then $W := \cup_{n \in \NN} B_{r_{n}}^{\RR^{d}}(2n)$ is a manifold 
   with boundary but \emph{does not have relative bounded geometry.}
\end{examples}

In the following section, we will position ourselves to use Theorem
\ref{thm.well.posedness}.

\section{Uniform estimates}
\label{sec.ue}
Our first results are in the framework of Babu\v{s}ka-Kondratiev 
(or weighted Sobolev) spaces, which we introduce next, after recalling
the most important points of our setting (notation and assumptions) 
from the introduction. In this section,
we still work in an arbitrary dimension $d$, except in the 
last part, that is, in \ref{ssec.wellp}, where we will assume $d = 2$. 

\subsection{Babu\v{s}ka-Kondratiev (weighted Sobolev) spaces}
\label{ssec.BKspaces}
We consider $d \in \NN$ arbitrary and the setting 
described in the introduction in Section \ref{ssec.setting}.
See, especially, assumption \ref{assumpt.intro1}.
In particular, the sets $\maV_{n} = \{p_{1n}, p_{2n}, \ldots, p_{kN}\} 
\subset \RR^{d}$ and the functions $r_{n}(x) := 
\eta(\dist_{\RR^{d}}(x, \maV_{n}))$ are as introduced there. Also, 
the manifold $(\Mce, \hat g) \ede \cup_{n \in \indexSet} 
   (\{n\} \times (\RR^{d}\smallsetminus \maV_{n}), \hat g_{n})$,
   $\hat g_{n} \ede r_{n}^{-2} dx^2$,
is as in Proposition \ref{prop.Mce.bg}. 
We shall also consider 
domains $\Omega_{n} \subset \RR^{d}$ 
satisfying
   $\Omega_{n} \cap \maV_{n} \seq \emptyset\,.$
We let
\begin{equation*}\label{eq.def.U}
   U \ede \cup_{n \in \indexSet} \{n\} \times \Omega_{n} 
   \subset \Mce \subset \indexSet \times \RR^{d}\,,
\end{equation*}
be as in Assumption \ref{assumpt.fw}, that is, $U$ is the \emph{disjoint} 
union of the domains $\Omega_{n}$.

\begin{remark}\label{rem.BKspaces} 
   Let $G \subset \Mce$ be an open subset.
   The Babu\v{s}ka-Kondratiev spaces on $G$ associated to 
   the weight $r$ are then
   \begin{equation*}
         \maK_{a}^m (G; r) \ede \{\, u : G \to \CC \mid 
         r^{|\alpha| - a} \pa^\alpha u \in L^2(G, dx^{2})\,, 
         \ |\alpha| \le m\, \}\,.
   \end{equation*}
   (Compare to the Introduction.) We endow the spaces $\maK_{a}^m (G; r)$
   with the resulting norms. Recall that   
   $\maK_{a}^m (\Omega_{n}; r_{n}) \ede \{\, u : \Omega_{n} \to \CC \mid 
   r_{n}^{|\alpha| - a} \pa^\alpha u \in L^2(\Omega_{n})\,, \ |\alpha| \le m\, \}$
   (Definition \ref{def.BKspaces}).
   It follows that
   $$\maK_{a}^m (U; r) \simeq \oplus_{n \in \indexSet} \maK_{a}^m (\Omega_{n}; r_{n})\,,$$ 
   (Hilbert space direct sum). The same relation holds also for $\Mce$.
\end{remark}

The following lemma is also well-known and explains why we are
interested in Sobolev spaces on manifolds. The following lemma 
follows, for instance, from the fact that $r$ is an admissible 
weight (i.e. $dr/r \in W^{\infty, \infty}(M; TM)$, where $TM$
is endowed with the metric $\hat g$, see Lemma 
\ref{cor.admissible.rn}).

\begin{lemma} \label{lemma.K=H}
   Let $\Mce \subset \indexSet \times \RR^{d}$ 
   be as in Proposition \ref{prop.Mce.bg}, in particular, its 
   metric is $\hat g = r^{-2}dx^2$. Then 
   \begin{enumerate}[(i)]
      \item We have an isomorphism of 
      normed spaces $ 
         H^{m}(\Mce; \hat g) \simeq \maK_{d/2}(\Mce; r)
      $ 

      \item Let us also assume that $U \ede \cup_{n \in \indexSet} 
      \{n\} \times \Omega_{n} \subset \Mce$ is a manifold with boundary 
      and relative bounded geometry, then 
      \begin{equation*}
         H^m(U; \hat g) \seq \maK_{d/2}^m (U; r) \,.
      \end{equation*}
   \end{enumerate}
\end{lemma}

\begin{proof} 
   The map 
   \begin{equation*} 
      W^{\infty, \infty}(\Mce)^{d} \ni 
      (f_{1}, f_{2}, \ldots, f_{d}) \to r(f_{1} \pa_{1} + f_{2}\pa_{1} + 
      \ldots f_{d} \pa_{d}) \in W^{\infty, \infty}(\Mce; \Mce)
   \end{equation*}
   is an isomorphism of Fréchet spaces. 
   In other words, the vector fields $r \pa_{j}$ form a topological 
   (or Fr\'echet) basis of all the vector fields on $\Mce$ with bounded 
   covariant derivatives. Hence the assumptions of Proposition 3.2 of 
   \cite{GN17} or Proposition 5.17(iv)
   \cite{KohrNistor1} are satisfied. (These vectors are also an orthonormal basis 
   at every point of $\Mce$, which is more than is required by these propositions.)
   Since $\Mce$ has bounded geometry,
   by Proposition \ref{prop.Mce.bg}, these propositions then give (the intuitively 
   obvious) first relation of the following sequence of equalities
   \begin{align*}
      H^{m}(\Mce; \hat g) & \seq \{u : \Mce \to \CC \mid (r \pa)^{\alpha} u 
      \in L^{2}(\Mce; \hat g)\,,\ \forall \, |\alpha| \le m \} \\
      & \seq \{u : \Mce \to \CC \mid 
      r^{-d/2}(r \pa)^{\alpha} u \in L^{2}(\Mce; dx^{2}) 
      \,,\ \forall \, |\alpha| \le m \} \\
      & 
      \seq \{u : \Mce \to \CC \mid 
      r^{|\alpha|-d/2}\pa^{\alpha} u \in L^{2}(\Mce; dx^{2})
      \,,\ \forall \, |\alpha| \le m \} \ =:\
      \maK_{d/2}^{m}(\Mce; r)\,.
   \end{align*}
   The second equality above follows from the fact that the volume form with 
   respect $\hat g$ is $r^{-d}$ times the volume form for $dx^{2}$ 
   and the third equality above follows from the Lemma \ref{lemma.diff.ops}.
   All the maps are continuous, and hence, by Banach's theorem, they are 
   isomorphisms of the corresponding normed spaces. The second part is proved
   in exactly the same way, but using only Proposition 3.2 of \cite{GN17}
   (which was proved for manifolds with boundary and relative bounded geometry)
   and the assumption that $U$ has relative bounded geometry (instead of 
   Proposition \ref{prop.Mce.bg}).
\end{proof}

See also \cite{AGN_CR}. The first part also follows immediately 
from the results in \cite{sobolev} because $\Mce$ is a Lie manifold.
In the same spirit, we shall need the following lemma (which is 
well-known for bounded domains).

\begin{lemma}\label{lemma.shift.K}
   Let $a, b \in \RR$. Then multiplication by $r^b$ induces an 
   isomorphism
   \begin{equation*} 
      r^b : \maK_{a}^m(U; r) \to \maK_{a+b}^m(U; r)\,,
   \end{equation*}
   whose norm depends continuously on $b$.
\end{lemma}

\begin{proof} 
   This also follows from Corollary \ref{cor.admissible.rn}.
   Indeed, the norm $u \in \maK_{a}^m(U; r)$ is 
   \begin{equation*}
      \|u\|_{\maK_{a}^m(U; r)}^{2} \ede \sum_{|\alpha| \le m}
      \|r^{|\alpha| - a} \pa^\alpha u \|_{L^2(U; dx^{2})}^{2}\,.
   \end{equation*}
   We have 
   \begin{equation*}
      r^{|\alpha| -a -b} \pa^\alpha (r^{b}u) \seq
      \sum_{\beta \le \alpha} {\alpha \choose \beta}
      r^{|\alpha-\beta|-b} \big ( \pa^{\alpha - \beta} r^{b} \big )
      r^{|\beta| - a} \pa^{\beta} u\,. 
   \end{equation*}
   Let us write $A \lesssim B$ when there is a $C > 0$
   independent of the relevant data (independent of $u$ 
   in the case below) such that $A \le CB$.
   Since $r^{|\alpha-\beta|-b} \pa^{\alpha- \beta} r^{b} \in 
   L^{\infty}(\Mce)$, by Corollary \ref{cor.admissible.rn}, 
   we successively obtain
   \begin{align*}
      \|r^{b} u\|_{\maK_{a+b}^m(U; r)}^{2} & \ede \sum_{|\alpha| \le m}
      \|r^{|\alpha| - a-b} \pa^\alpha (r^{b}u) \|_{L^2(U;  dx^{2})}^{2}\\
      & \ \lesssim\  \sum_{\beta \le \alpha} 
      \|r^{|\alpha-\beta|-b} \pa^{\alpha- \beta} r^{b}\|_{L^\infty}
      \|r^{|\beta| - a} \pa^{\beta} u\|_{L^2}
      \ \lesssim\  \|u\|_{\maK_{a}^m(U; r)}^{2}\,.
   \end{align*}
   Using the continuity in $b$ in Corollary \ref{cor.admissible.rn}
   and by separating the terms with $\beta = 0$ from the other ones,
   the same calculation gives
   \begin{equation*}
      \|r^{b} u\|_{\maK_{a+b}^m(U; r)}^{2} \le \|u\|_{\maK_{a}^m(U; r)}^{2}
      + C |b|\,.
   \end{equation*}
   (because all the terms with $\beta \neq 0$ contain a factor of $b$),
   with $C$ independent of $a$ (but possibly depending on $\alpha$).
   Replacing $u$ with $r^{-b}u$, using the continuity of multiplication 
   with $r^{-b}$ and then adjusting the indices, we
   obtain 
   \begin{equation*}
      \|u\|_{\maK_{a}^m(U; r)}^{2} \le \|r^{b} u\|_{\maK_{a+b}^m(U; r)}^{2}
      + C |b|\,,
   \end{equation*}
   and hence the continuity of the norm at $b=0$. The continuity at other 
   values follow from the continuity at 0 by reparametrization in $a$ and 
   $b$.
\end{proof}

\subsection{Uniform estimates in Babu\v{s}ka-Kondratiev spaces}
\label{ssec.wellp}

We now turn to the precise statements and proofs of Theorem \ref{thm.wellp1}
and its Corollary \ref{cor.unif1}, which will be, to a large extent, 
applications of the well-posedness result in the previous section, namely 
Theorem \ref{thm.well.posedness}.  
\textbf{From now on $d = 2$.
In addition to Assumption \ref{assumpt.intro1}, 
we also assume from now on Assumption \ref{assumpt.fw},}
as well as all the notation in Section \ref{ssec.setting}.
We are ready now to prove our first main result, which is 
Theorem \ref{thm.wellp1} for $a = 0$. Its assumptions 
are the usual ones, as recalled in its statement next.

\begin{theorem}\label{thm.global.isomorphism}
   We continue to assume \ref{assumpt.intro1} and \ref{assumpt.fw}, 
   in particular, the manifold with boundary $(U, \hat g) := \cup_{n \in \indexSet} 
   (\Omega_{n}, r_{n}^{-2}dx^2)$ has finite width (Definition \ref{def.finite.width}).
   Let us also assume $d = 2$, so $\Omega_{n} \subset \RR^{2}$,
   then the Dirichlet Laplacian induces isomorphisms
   \begin{equation*}
      \Delta_{D, m} \ede \Delta : \maK_{1}^{m+1}(U; r) 
      \cap \{ u \vert_{\pa U} = 0\} 
      \to \maK_{-1}^{m-1}(U; r)\,, \quad m \in \ZZ_+\,.
   \end{equation*}
\end{theorem}

\begin{proof}
   Let $\Delta_{\hat g}$ be the Laplace operator on $U$ associated to the 
   metric $\hat g$. Let $H^j(U, \hat g)$ be the Sobolev space of order $j$ defined by 
   the metric $\hat g$, as before. It follows from Theorem \ref{thm.well.posedness} 
   that, for any $m \in \ZZ_{+}$, 
   \begin{equation}\label{eq.isom.H}
      \Delta_{\hat g} : H^{m+1}(U, \hat g) \cap H_0^1(U; \hat g) \to H^{m-1}(U, \hat g)\,,
   \end{equation}
   is an isomorphism. As in Lemma \ref{lemma.K=H}, we have isomorphisms
   \begin{equation}\label{eq.conf.mod}
      H^{m}(U, \hat g) \simeq \maK_{1}^m(U; r)\,,
   \end{equation}
   the index one in $\maK_{1}^m$ is half of the ambient dimension 
   (which is 2). 
   We also have that $\Delta_{\hat g} = r^2 \Delta_{D, m}$ (this is specific to 
   dimension 2). The proof is completed by combining the isomorphisms 
   of equations \eqref{eq.isom.H} and \eqref{eq.conf.mod} with that 
   of Lemma \eqref{lemma.shift.K} with the relation $\Delta_{\hat g} = 
   r^2 \Delta_{D, m}$.
\end{proof}

Given the direct sum decomposition 
$\maK_a^m(U; r) \simeq \oplus_{n \in \indexSet}
\maK_a^m(\Omega_{n}; r_{n})$ (Hilbert space direct sum, see 
Remark \ref{rem.BKspaces}) 
and that the Laplacian $\Delta$ preserves
this decomposition, we obtain the following equivalent form of the 
above theorem, which we formulate as a corollary.

\begin{corollary}\label{cor.isomorphisms}
   We keep the assumptions and notation of Theorem \ref{thm.global.isomorphism}.
   Then the Dirichlet Laplacian $\Delta_D$ induces an isomorphism
   \begin{equation*}
      \Delta_{D, m, n} \ede \Delta : \maK_{1}^{m+1}(\Omega_{n}; r_{n}) 
      \cap \{ u \vert_{\pa \Omega_{n}} = 0\} 
      \to \maK_{-1}^{m-1}(\Omega_{n}; r_{n})\,, \quad m \in \ZZ_+\,,
   \end{equation*}
   such that the norms of the operators 
   $(\Delta_{D,m,n})^{-1} : \maK_{-1}^{m-1}(\Omega_{n}; r_{n}) \to 
   \maK_{1}^{m+1}(\Omega_{n}; r_{n})$ are uniformly bounded in 
   $n \in \indexSet$.
\end{corollary}

A perturbation argument applied to the continuous 
family $r^{a}\Delta_D r^{-a}$ then yields the proof 
of Theorem \ref{thm.wellp1}. We give here a slightly stronger
formulation. We also use the formulation obtained
by decomposing our spaces as a direct sum 
in $n \in \indexSet$, as in the previous corollary.

\begin{theorem}
   \label{thm.etaU} 
   We keep the assumptions and notation of Theorem 
   \ref{thm.global.isomorphism}, so, in particular, $d =2$.
   Then there exists $\delta_{U} > 0$ with the following property. 
   For any $m \in \ZZ_+$ and any $|a| < \delta_{U}$, 
   the Dirichlet Laplacian $\Delta_D$ induces isomorphisms
   \begin{equation*}
      \Delta_{D, m, n, a} \ede \Delta : \maK_{a+1}^{m+1}(\Omega_{n}; r_{n}) 
      \cap \{ u \vert_{\pa \Omega_{n}} = 0\} 
      \to \maK_{a-1}^{m-1}(\Omega_{n}; r_{n})
   \end{equation*}
   such that the norms of the operators $(\Delta_{D,m,n,a})^{-1} : 
   \maK_{a-1}^{m-1}(\Omega_{n}; r_{n}) \to \maK_{a+1}^{m+1}(\Omega_{n}; r_{n})$ 
   are uniformly bounded in $n \in \indexSet$ and
   $a \in [-\gamma, \gamma] \subset (-\delta_{U}, \delta_{U})$.
\end{theorem}

\begin{proof}
   Only the indices $n$ and $a$ make a difference, so we shall write 
   $\Delta_{n, a}$ instead of $\Delta_{D,m,n,a}$ ($m$ is implicit and 
   ``D'' is only to remind us that we are using Dirichlet boundary 
   conditions). Also, recall that only the domains of the operators 
   $\Delta_{n, a}$ differ, their form does not.
   At distance $\ge R/5$ to $\maV_{n}$, $r$ is constant and hence we have 
   $r^{-a}\Delta_{n,a} r^{a} = r_{n}^{-a}\Delta_{n,a} r_{n}^{a} = \Delta_{n, a}$.
   Next, let us assume we are on $\Omega_{n} \equiv \{n\} \times \Omega_{n}$,
   $n \in \indexSet$ at distance $\le R/5$ to one of the vertices 
   $P_{jn} \in\maV_{n}$. (This vertex is then uniquely determined,
   by our assumptions, see point \eqref{item.intro.as3}
   in Section \ref{ssec.setting}.) 
   Let $\rho$ be the distance function to $p_{jn}$.
   Then $r_{n} = \eta(\rho)$. Let 
   $$A \ede \frac{\rho \eta'(\rho)}{\eta(\rho)}\,.$$
   Our assumptions on $\eta$ imply that $A$ has support inside the ball 
   of radius $R/5$ and that, near $0$, it is equal to $1$. Hence 
   $A$ extends to a function all of whose derivatives are bounded, and 
   hence to a multiplier of all the spaces $\maK_a^m(U; r)$.
   (Equivalently, $A$ defines a bounded operator on each 
   $\maK_a^m(\Omega_{n}; r_{n})$
   and the norms of these operators are bounded uniformly in 
   $n \in \indexSet$.)
   In polar coordinates $(\rho, \theta)$ near that vertex, we have 
   \begin{equation}\label{eq.pert}
      \begin{gathered}
      r^{-a}\Delta_{n, a} r^{a} - \Delta_{n, 0} 
      \seq \rho^{-2}\big( (\rho\pa_\rho + aA)^2 + \pa_\theta^2) - 
      \rho^{-2}\big( (\rho\pa_\rho)^2 + \pa_\theta^2)\\
      \seq \rho^{-2} \big [ 2a A \rho \pa_\rho + a \rho \pa_\rho(A) 
      + a^2A^2]
      \end{gathered}
   \end{equation}
   Since the operators $A : \maK_b^m(U; r) \to \maK_b^m(U; r)$ and 
   $\rho \pa_\rho  : \maK_b^m(U; r) \to \maK_b^{m-1}(U; r)$ are bounded, 
   we obtain that the norm 
   $$\|r^{-a}\Delta_{n, a} r^{a} - \Delta_{n, 0}\|_{\maK_1^{m+1}(U; r), 
   \maK_{-1}^{m+1}(U; r)} \le C_m a\,,$$
   with $C_m$ independent of $n \in \indexSet$ and 
   $a \in [-\gamma, \gamma]$
   (but depending on $m$ and $\gamma$). 

   We know from Theorem \ref{thm.global.isomorphism}
   that $\Delta_{n, 0}$ is invertible and the norm of its inverse 
   is independent of $n \in \indexSet$. Hence, for $a$ small enough
   ($|a| < \delta_{U, m}$, depending on $m$ at this time), 
   $r^{-a}\Delta_{n, a} r^{a}$ will also be invertible. This amounts 
   to the claimed invertibility of $\Delta_{n, a}$ in view of 
   the fact that $r^{a} : \maK_{1}^{m}(U; r) \to \maK_{a +1}^{m}(U; r)$
   is an isomorphism whose norm as well as that of its inverse depend 
   continuously on $a$ (Corollary \ref{cor.admissible.rn}). 
   
   We are just left to prove that we can choose the bound for $a$
   independently of $m$. Indeed, let $\delta_{U, 0}$
   be such that, for all $|a| < \delta_{U, 0}$, the map 
   $\Delta_{n, a} \ede \Delta : 
   \maK_{a+1}^{1}(\Omega_{n}; r_{n}) \cap \{ u \vert_{\pa \Omega_{n}} = 0\} 
   \to \maK_{a-1}^{-1}(\Omega_{n}; r_{n})$
   is invertible. Let $f \in \maK_{a-1}^{m-1}(\Omega_{n}; r_{n})$. 
   There exists then $u \in \maK_{a+1}^{1}(\Omega_{n}; r_{n}) 
   \cap \{ u \vert_{\pa \Omega_{n}} = 0\}$ such that 
   $\Delta u = f$. The formula \eqref{eq.pert} and the 
   invertibility of $\Delta_{n, a}$ (for all $m$) show
   by induction on $m$ that, in fact, 
   $u \in \maK_{a+1}^{m+1}(\Omega_{n}; r_{n})$. (This also follows 
   also directly from the regularity results of \cite{GN17}, 
   this method is explained in more detail in \cite{DLN2}.) Consequently,
   $\delta_{U, m} = \delta_{U, 0}.$
\end{proof} 

By using the Hardy inequality instead of the Poincar\'e 
inequality, we can extend these results to domains 
$\Omega \subset \RR^{d}$, for any $d \ge 2$. If our domain 
has no cracks, the same 
argument will give also the more explicit bound $\delta_{U}
= \frac{\pi}{\alpha_{MAX}}$ \cite{DLN_CR, DLN2} for $d = 2$,
where $\alpha_{MAX}$ is the largest angle of $\Omega_{\infty}$.
The assumption that the index set $\indexSet$
be countable is not really necessary for the results of this 
section (but it simplifies the presentation, especially for 
Sobolev spaces).

\subsection{Returning to the usual Sobolev spaces}
\label{ssec.returning}

Recall that a \emph{contraction} is a continuous linear map of norm 
$\le 1$ between normed spaces. We shall use the following lemma
(in which it is crucial that we have chosen the weights $r_{n} \le 1$).
We use the notation and hypotheses of Assumptions \ref{assumpt.intro1}
and Assuptions \ref{assumpt.fw}, which were briefly recalled 
in the previous section.

\begin{lemma} \label{lemma.contraction}
   For all $m \in \ZZ_{+}$ and $n \in \indexSet$,
   the maps $\maK_{m}^m(\Omega_{n}; r_{n}) \to H^m(\Omega_{n})$
   are continuous of norm $\le 1$ (i.e. \emph{contractions}).
   Similarly, if $a \le b$, then the map $\maK_b^m(\Omega_{n}; r_{n}) 
   \to \maK_a^m(\Omega_{n}; r_{n})$ is a contraction.
\end{lemma}

\begin{proof} Recall that $\eta \le R/6 \le 1/6$. Then, 
   for any multi-index $\alpha$ with $|\alpha| \le m$, we have 
   $r_{n}^{|\alpha| - m} \ge 6^{m-|\alpha|} \ge 1 .$
   Hence $$\|r_{n}^{|\alpha| - m} \pa^\alpha u\|_{L^2(\Omega_{n})} \ge 
   \|\pa^\alpha u\|_{L^2(\Omega_{n})}\,,$$
   which gives $\|u\|_{H^m(\Omega_{n})} \le 
   \|u\|_{\maK_m^m(\Omega_{n})}$, as desired.
   The second part follows similarly from the fact that 
   $r_{n} \le 1$.
\end{proof}

We need one final lemma before being able to relate our results
in weighted spaces to the usual Sobolev spaces.
Given three Banach spaces $X, Y, Z$, $X \cup Y \subset Z$,
we shall denote by $[X, Y]_{\theta}$ their complex interpolation 
space, where $\theta \in (0, 1)$.

\begin{lemma}\label{lemma.interpolation}
   Let $G \subset \RR^{d}$ be an open subset, let $f > 0$ be a 
   continuous function on $G$, and let $\theta \in (0, 1)$. We then 
   have the isomorphism 
      $$ [\maK_{a}^{m}(G; f), \maK_{b}^{m}(G; f)]_{\theta}
      \simeq \maK_{(1-\theta)a + \theta b}^{m}(G; f).$$
\end{lemma}

We obtain the following needed corollary.

\begin{corollary} \label{cor.contraction}
   For all $a\in [0, 1]$ and $n \in \indexSet$,
   the map $\maK_{1 + a}^2(\Omega_{n}; r_{n}) \to H^{1 + a}(\Omega_{n})$
   is well defined and continuous of norm $\le 1$. 
\end{corollary}

\begin{proof} 
   For $a = 1$, the result is simply Lemma \ref{lemma.contraction} for 
   $m = 2$. The map $\maK_1^2(\Omega_{n}) \to \maK_1^1(\Omega_{n}) 
   \to H^1(\Omega_{n})$ is a contraction as composition of contractions 
   (again by Lemma \ref{lemma.contraction}). This gives the result for 
   $a = 0$. For $a \in (0, 1)$, the result follows by interpolation,
   see Lemma \ref{lemma.interpolation}, since our spaces $H^{s}(G)$
   are defined by interpolation. 
\end{proof}

Using the continuous map $\maK_{1+\delta}^2(\Omega_{n}) 
\to H^{1 + \delta}(\Omega_{n})$,
we obtain the desired result for $\delta := a - \delta_{U}$,
as follows.

\begin{theorem}\label{theorem.main}
   Let $\Omega_{n} \subset \RR^{2}$ be a sequence of 
   domains as in Theorem \ref{thm.etaU} (i.e. 
   satisfying assumptions \ref{assumpt.intro1}
   and \ref{assumpt.fw}). Let $\delta_{U} > 0$ be as in 
   Theorem \ref{thm.etaU} and let $|a| < \delta_{U}$, $a \le 1$. 
   Then there exists $C_H > 0$ such that, for any 
   $f_{n} \in L^2(\Omega_{n})$, $n \in \indexSet$,
   \begin{equation*}
      \|\Delta_{D, n}^{-1}f_{n}\|_{H^{1+a}} 
      \le C_H \|f_{n}\|_{L^2(\Omega_{n})}\,.
   \end{equation*}
\end{theorem}

\begin{proof} 
   Let $C_{K}$ be the bound of the inverses of the 
   operators $$\Delta_{D, n} : \maK_{1+a}^2(\Omega_{n}; r_{n}) 
   \cap \{u\vert_{\pa \Omega_{n}} = 0 \} \to 
   \maK_{-1+a}^0(\Omega_{n}; r_{n})$$
   in Theorem \ref{thm.etaU}, $n \in \indexSet$.
   We have $L^2(\Omega_{n}) = \maK_0^0(\Omega_{n}; r_{n}) 
   \to \maK_{-1 + a}^{0}(\Omega_{n}; r_{n})$,
   where the last map is a contraction (since $a-1 \le 0$). 
   Hence, combining this discussion with Corollary \ref{cor.contraction}, 
   we obtain
   \begin{equation*}
      \|\Delta_{D, n}^{-1}f_{n}\|_{H^{1+a}(\Omega_{n})} 
      \le \|\Delta_{D, n}^{-1}f_{n}\|_{\maK_{1+a}^2(\Omega_{n})}
      \le C_K \|f_{n}\|_{\maK_{a-1}^0(\Omega_{n})}\\
      \le C_K \|f_{n}\|_{L^{2}(\Omega_{n})}\,.
   \end{equation*}
   This completes the proof with $C_{H} = C_{K}$.
\end{proof}

\section{Approximating a polygon}
\label{sec.Benoit}

We now construct an example of a family of domains $\Omega_{n}$,
$n \in \indexSet = \oN$, for which our assumptions (and hence 
our results) are satisfied.

More precisely, given a straight polygonal domain $\Omega_{\infty}$,
we now construct a completely explicit sequence of smooth,
bounded domains $\Omega_{n},$ $n \in \NN$, such that the 
resulting manifold with boundary $(U, \hat g) := 
\cup_{n \in \oN} (\Omega_{n}, \hat g_{n})$ is of relative bounded geometry
and, in a certain sense, the domains $\Omega_{n}$, $n \in \NN$,
converge to $\Omega_{\infty}$. For this choice of domains, 
we check that all the conditions of Assumptions 
\ref{assumpt.intro1} and \ref{assumpt.fw} are satisfied.
(So, from now on, we shall use the \emph{notation} of those assumptions, 
but we do not assume them to be satisfied; we rather prove
that they are satisfied).

In this section, we continue to assume $d = 2$ 
(so our domains $\Omega_{n}$ will be in the plane),
as in the last two sections, Sections \ref{ssec.wellp} 
and \ref{ssec.returning}. We shall also assume 
that $\alpha_{MAX} < 2\pi$.

\subsection{Curvilinear polygonal domains}
We provide now a rather general definition of a 
``curvilinear polygonal domain.''

\begin{definition}\label{def.polygon}
   Let $M$ be a two-dimensional smooth manifold and $D \subset M$. 
   We shall say that $D$ is \emph{a curvilinear polygonal domain}
   in $M$ if the following conditions are satisfied.
   \begin{enumerate}[(i)]
      \item $D$ is open in $M$ and $\overline{D}$ is compact.

      \item There exist a finite set of points 
      $$\maV \ede \{q_1, q_2, \ldots , q_N \} \subset \pa D 
      \ede \overline{D} \smallsetminus D$$
      and $\maC^{1}$-curves $\gamma_{k} : [0, 1] \to M$, 
      $k = 1, \ldots, N$, such that $\gamma_{k}(0) = q_{k}$ and 
      $\gamma_{k}(1) = q_{k+1}$ (we set $q_{N+1} = q_{1}$), satisfying
      also the following conditions.
      
      \item The sets $\gamma_{k} \big((0, 1) \big)$ are disjoint. 

      \item $\pa D \smallsetminus \maV = \cup_{k=1}^N 
      \gamma_{k}\big((0, 1) \big).$

      \item For each $k$, $\gamma_{k}$ is smooth on $(0, 1)$.
   \end{enumerate}
\end{definition}

Thus, the sets $\gamma_{k}((0, 1))$ together with $\maV$, form a 
partition of the boundary $\pa D$ of $D$ and 
$\pa D = \cup_{k=1}^N \gamma_{k} \big([0, 1] \big)$. Note that we 
\emph{do not} assume that any of the $\gamma_{k}$ is contained in 
a straight line, hence the terminology ``curvilinear polygonal domain.''
Note also then that the sides (or edges) $\gamma_{k}([0, 1])$ of $D$ do not
have self-intersections (they may intersect only in $\maV$). See also
\cite{DaugeBook, Kondratiev67, KMR}. More general polygonal 
domains were used in \cite{withLi} (in that paper, 
cracks and self-intersections were allowed, the point being that 
these ``degeneracies'' did not affect the optimal rates of convergence 
in the Finite Element Method).

\begin{remark}
The set $\maV$ of Definition \ref{def.polygon} is \emph{not} 
determined uniquely by $D$ (in fact, the 
set $\maV$ can always be enlarged). However, there will always 
be a \emph{minimal} such set, which we will
denote by $\maV_D$; it is uniquely determined by the condition
that, for all $k = 1, \ldots, N$, 
$\gamma_{k}$ and $\gamma_{k+1}$ cannot be joined in a smooth 
curve (we set $\gamma_{N+1} = \gamma_{1}$). The elements 
of this minimal set $\maV_D$ are called \emph{the vertices} of $D$. 
\end{remark}

Given two points $A, B \in \RR^{d}$ ($d$ arbitrary here),
we shall let $[AB]$ denote the \emph{straight line segment}
joining $A$ and $B$. Thus $[AB]$ is a compact subset of $\RR^{d}$.
We shall use this for $d = 2$ to say that $D$ is a 
\emph{straight polygonal domain} if, for all $k$,
$\gamma_{k}([0, 1]) = [q_{k}q_{k+1}]$. That is, all edges of 
$D$ are straight line segments.

\subsection{A distinguished sequence approaching a straight polygon}
\label{ssec.Simon}
Let $\Omega_{\infty}$ be a straight polygonal domain in $\RR^{2}$.
We shall construct in this section an explicit example 
of a sequence $(\Omega_{n})_{n \in \NN}$ of smooth domains such that 
the family $\Omega_{n}$, $n \in \indexSet := \oN$, satisfies the 
assumptions needed so that the results of Section 
\ref{ssec.wellp} apply to it. Recall from the introduction that, in addition to 
the domains $\Omega_{n}$, we need to construct also sets 
$\maV_{n} = \{p_{1n}, p_{2n}, \ldots, p_{Nn}\}\in 
\RR^{2} \smallsetminus \Omega_{n}$ satisfying 
certain additional assumptions (in this section, the sets $\maV_{n}$
will be used to define the domains $\Omega_{n}$, for $n \in \NN$). 
We proceed in order.

\subsubsection{The sets $\Omega_{\infty}$ and $\maV_{\infty}$}
The domain $\Omega_{\infty}$ is our given \emph{straight 
polygonal domain} in $\RR^2$ and $\maV_{\infty}$ is its set of 
vertices
\begin{equation}\label{eq.def.V0}
   \maV_{\infty} \ede \maV_{\Omega_{\infty}} \seq 
   \{ p_{1}, p_{2}, \ldots , p_{N} \} 
   \subset \pa \Omega_{\infty} 
   \ede \overline{\Omega}_{\infty} \smallsetminus \Omega_{\infty}\,.
\end{equation}
Its edges are $[p_{j}p_{j+1}]$, where we recall that we agreed that 
$p_{N+1} = p_{1}$. (See the previous section for the other notation.)
Notice also that we write $p_{j}$ instead of $p_{j\infty}$,
for the simplicity of the notation.
\smallskip

Let us now turn to \emph{the construction of the other sets 
$\Omega_{n}$ and $\maV_{n}$.} To simplify the notation,
we shall write $\ball{x}{a}$ for the ball of radius $a$ and 
center $x$ in $\RR^{2}$ (instead of $B_{a}^{\RR^2}(x)$, which
is the typical notation in this paper, see Equation 
\eqref{eq.def.ball.r}). Also, 
in the following, we will generally assume $j \in \{1, \ldots, N\}$.
Thus, when we will say ``for each $j$'', we will mean 
``for each $j \in \{1, \ldots, N\}$''. Here the order is 
the ``circular order,'' meaning that, 
occasionally, we will also use objects associated to $N+1$.
In that case, they are the same as the objects associated
to $1$. For instance, above we have used $p_{N+1} = p_{1}$.

\subsubsection{The sets $\Omega_{1}$ and $\maV_{1}$}
Let $R_{0}$ be the \emph{minimum} distance between any two
\emph{non-intersecting} (closed) edges of $\Omega_{\infty}$ and 
$\rho \in (0, R_{0}/2)$. This ensures that the distance between any two 
edges of $\partial\Omega_{\infty}$ with \emph{no common vertices} 
is larger than $2\rho$. Let $\mathbf{b_j}$ be the
\emph{exterior} bisectrix of the angle of 
$\Omega_{\infty}$ at $p_{j}$. Thus $\mathbf{b_j}$ is a 
half-line starting at $p_{j}$. Our construction proceeds 
as follows.\smallskip

$\bullet$\ Let $p_{j1}$ be the \emph{unique} point 
on the bisectrix $\mathbf{b_j}$ at distance
$\dist_{\RR^2}(p_{j},p_{j1}) = \rho/2$ to $p_{j}$. Then
\begin{equation}\label{eq.def.Simon.V}
   \maV_{1} \ede \{p_{11}, p_{21}, \ldots, p_{N1}\}\,.
\end{equation}

To define the set $\Omega_{1}$, we first define its 
boundary, which will be a smooth curve obtained by 
piecing together several other (smaller) curves. Some 
of these curves will be
straight segments, which we define first.

$\bullet$\ \emph{The segments $\ell_{j1}$ 
defining the straight part of $\pa \Omega_{1}$
are constructed as follows.}
For $\rho' \in (0, \rho/4]$ small enough, we have that,
for each $j$, the circle 
$\partial \ball{p_{j1}}{\rho/2}$ will contain 
\emph{exactly two points} $q_{j1}$ and $q'_{j1}$ 
that are at distance $\rho'$ to $\Omega_{\infty}$.
In particular, $q_{j1}$ and $q'_{j1}$ are \emph{outside} 
$\Omega_{\infty}$. We choose $q_{j1}$ to be on the same
side of the bisectrix $\mathbf{b_j}$ as $p_{j-1}$ (that is, 
in the same half-space whose boundary contains $\mathbf{b_j}$). 
By symmetry, $q'_{j1}$ and $p_{j+1}$ will also be on the 
same side of $\mathbf{b_j}$. Then 
$$\ell_{j1} \ede [q_{j1}'q_{(j+1)\,1}]\,,$$ 
the straight line segment with extremities $q_{j1}'$ and 
$q_{(j+1)\,1}$. Recall, in this context, that 
$q_{(N+1)\, 1}:= q_{11}$, $q_{(N+1)\, 1}:= q_{11}'$,
and $\ell_{(N+1)\,1} = \ell_{11}$,
as explained above. We also choose $\rho'$ small enough
so that $\ell_{j1} := [q_{j1}'q_{(j+1)\,1}]$ is 
parallel to the edge $[p_{j}p_{j+1}]$ of $\Omega_{\infty}$.
(More precisely, we choose $\rho'$ small enough so that 
the distance from $q_{j1}$ to $\Omega_{\infty}$ is the foot 
of the perpendicular from $q_{j1}$ to the line containing
$[p_{j-1}p_{j}]$. This is possible because the balls 
$\ball{p_{j1}}{\rho/2}$ are disjoint, by the definition 
of $\rho$, because $\ball{p_{j1}}{\rho/2}\subset \ball{p_j}{\rho}$,
and because all points of the ball $\ball{p_{j1}}{\rho/2}$ are
at distance at least $\rho$ to all the edges of $\Omega_{\infty}$,
except the ones containing $p_{j}$.)

$\bullet$\ \emph{We now define the curves $c_{j1}$
forming the non-straight part of the boundary.} These curves 
$c_{j1}$ smoothly join the segments $l_{j1}$. More precisely, 
let $$c_{j1} \subset \ball{p_{j1}}{\rho/2}$$ 
be a smooth curve joining $q_{j1}$ and $q'_{j1}$,  
separating $p_j$ and $p_{j1}$ inside $\ball{p_{j1}}{\rho/2}$, 
and such that $\ell_{1\,j-1}\cup c_{j1} \cup\ell_{j1}$ is smooth 
without self-intersections (recall that $\ell_{(N+1)\,1} = \ell_{11}$).

$\bullet$\ Finally, we define the domain $\Omega_{1}$ to be 
the unique bounded domain with 
\begin{equation}\label{eq.def.Simon.Omega}
   \pa \Omega_{1} \seq \cup_{j=1}^{N} 
   \big (\ell_{j1} \cup c_{j1} \big)\,.
\end{equation}
Then $\Omega_{1}$ is a smooth, bounded domain. (This is
because its boundary is the union of all the sets 
$\ell_{j1}$ and $c_{j1}$, which were defined such 
that their union is locally smooth, and hence smooth.)

\subsubsection{The sets $\Omega_{n}$ and $\maV_{n}$, $n \ge 2$}
These sets are defined in the same way $\maV_{1}$ and 
$\Omega_{1}$ were defined, but using $\rho/n$ and $\rho'/n$
instead of $\rho$ and $\rho'$ and, for each $j$, choosing the curves 
$c_{jn}$ to be \emph{homothetic} to $c_{j1}$.

\subsubsection{Remarks}
Let us make a few simple remarks. The first one details 
the construction of the sets $\Omega_{n}$ and $\maV_{n}$
for $n \ge 2$. We shall use homotheties.

\begin{notation}\label{not.homothety}
   For $p\in\mathbb{R}^2$ and $c>0$, let 
   $h_{p,c}:\mathbb{R}^2\to\mathbb{R}^2$ 
   denote the homothety of center $p$ and ratio $c$, that is 
   $$h_{p, c}(q) \ede p + c(q-p)\,.$$
\end{notation}

It follows from definitions that $h_{p,c}$ 
is an isometry on $\left(\RR^2\setminus\{p\},
\frac{dx^2}{|x-p|^2}\right)$.

\begin{remark}\label{rem.details.nGE2}
   Let $n\geq2$ be fixed. Our construction is thus 
   such that $\Omega_{n}\subset\mathbb{R}^2$ is 
   a domain containing $\overline\Omega_{\infty}$ with boundary
   $\partial\Omega_{n}$ an embedded smooth curve consisting 
   of line segments $\ell_{1n},\dots,\ell_{Nn}$ and curves 
   $c_{1n},\dots c_{Nn}$:
   \begin{equation}\label{eq.def.Simon.Omegan}
      \pa \Omega_{n} \seq \cup_{j=1}^{N} 
      \big (\ell_{jn} \cup c_{jn} \big)
   \end{equation}
   such that
   \begin{itemize}
      \item $\ell_{jn} = [q'_{jn}, q'_{(j+1)\,n}]$,
      where $q_{jn}' = h_{p_j,1/n}(q'_{j1})$ and 
      $q_{jn} = h_{p_{j},1/n}(q_{j1})$ and 

      \item $c_{jn}=h_{p_j,1/n}(c_{j1})$.
   \end{itemize}
   In particular $\ell_{jn}$ is parallel 
   to the edge $[p_jp_{j+1}]$ of $\Omega_{\infty}$ and 
   at distance $\rho'/(4n)$ to it.
   Finally, we have $p_{jn} := h_{p_j,1/n}(p_{j1})$
   and $\maV_{n} :=\{p_{1n}, p_{2n}, \dots,p_{Nn} \}$.
\end{remark}

The following remark summarizes some additional properties of
$\Omega_{n}$ and $\maV_{n}$.

\begin{remark}\label{rem.properties}
Let $\mathcal{V}_{n}=\{p_{1n},\dots,p_{Nn}\}$, as above. 
   \begin{enumerate}[(i)]
      \item 
      $\lim_{n \to \infty} p_{nj} = p_j$ for each $j\in\{1,\dots,N\}$,
      \item $\maV_{n} \cap \Omega_{n} = \emptyset$ for each $n\in\NN^*$,
      \item For each fixed $n$, the distance between any two points 
      of $\maV_{n}$ is $\ge R_{0} - \rho \ge R_{0}/2$,
      so Assumption \ref{assumpt.intro1}\eqref{item.intro.as3} 
      is satisfied with $R = R_{0}/2$.
   \end{enumerate}
\end{remark}

Finally, the following remark explains the need for $\rho'$
(the smallest among the many parameters used).

\begin{remark}\label{rem.alpha} 
   Let $\alpha_{MAX}$ and $\alpha_{min}$ be the maximum, 
   respectively, minimum angles of $\Omega_{\infty}$. Of course, 
   $\alpha_{min}, \alpha_{MAX} \in (0, 2\pi)$. For $\rho$ 
   fixed (but small), we have $\rho' \to 0$ as $\alpha_{min} \to 0$ 
   or $\alpha_{MAX} \to 2\pi.$ So we cannot choose 
   $\rho'$ to be some constant multiple of $\rho$.
\end{remark}

\subsection{Proof that the `distinguished' $U$ has finite width}
We use the notation and constructions introduced in the previous 
sections (see, for example, Assumptions \ref{assumpt.intro1} and 
\ref{assumpt.fw}). However, we stress that, in this section,
we are \textbf{not assuming} \ref{assumpt.intro1} and 
\ref{assumpt.fw}, but rather prove that they are satisfied by our
constructions. 

Let us recall first some of the needed constructions and notations
from Subsection \ref{ssec.setting}. First, we endow each 
$\Omega_{n}$, $n \in \indexSet$, with 
the metric $\hat g_{n} = r_{n}^{-2} (dx_1^2 + dx_2^2)$, where 
$dx_1^2 + dx_2^2$ is the usual (euclidean) metric on $\RR^2$
and $r_{n}(x) := \eta_{\RR^{2}}(\dist(x, \maV_{n}))$, with 
$\eta$ is our fixed function of Assumption~\ref{assumpt.intro1}.
Let $U$ be the \emph{disjoint union} of the domains $\Omega_{n}$, 
$n \in \indexSet$, that is
   $$U \ede \cup_{n \in \indexSet} \{n \} \times \Omega_{n} 
   \subset \indexSet \times \RR^2 \,,$$
as in Equation \eqref{eq.def.U}.
We endow $U$ with the metric $\hat g$ that on the component $\Omega_{n}$ 
coincides with $\hat g_{n}$. Thus, if $r(n,x)=r_{n}(x)$ for $n\in\NN$ and 
$x\in\Omega_{n}$, then 
\begin{equation}\label{eq.conf.metric}
   \hat g \seq r^{-2} (dx_1^2 + dx_2^2)\,.
\end{equation}
If the sequence $\Omega_{n}$ is as in the previous section
(Section \ref{ssec.Simon}), then we shall call the associated
$U$ the \emph{distinguished} $U$.

We will now show that the geodesic curvature of $\pa U$ and it 
derivatives of all orders are bounded. This is enough in view of 
Remark \ref{rem.2for2}. We choose $\rho = R/2$ and 
rescale such that $4\rho = R = 1$. (So $R_{0} \ge 2$.)
We shall need the following remark.

\begin{remark}\label{rem.2for2}
   Let $M$ be a Riemannian manifold of dimension $d=2$, 
   then $H$ is a curve. Let $s$ be a choice of arc-length coordinate on $H$, 
   let $\pa_s$ be the associated unit length tangent vector field 
   to $H$, and let $\nu$ the unit normal to $H$. Then $g(\pa_{s}, \nu) = 0$
   and $g(\nu,\nabla^H_{\pa_s}\pa_s)=0$.
   Therefore, the second fundamental form $\II^{H}$ of $H$ is given by
   \begin{equation*}
      \II^{H}(\pa_s,\pa_s) \seq 
      g(\nu,\nabla^M_{\pa_s}\pa_s-\nabla^H_{\pa_s}\pa_s)
      \seq g(\nu,\nabla^M_{\pa_s}\pa_s) \seq -
      g(\pa_s,\nabla^M_{\pa_s}\nu)\, =:\, \kappa(s),
   \end{equation*} 
   that is, $\II^{H}(\pa_s,\pa_s)$ coincides with the geodesic curvature 
   $\kappa(s)$ of $H$.
\end{remark}

\begin{lemma}\label{lemma.curvature}
   Let $\sigma$ be a smooth function on an open set $U\subset\RR^2$.
   The geodesic curvature $\mathring\kappa$ of a curve in $U$ in the 
   Euclidean metric $dx^2$ and its geodesic curvature $\kappa$ in the 
   metric $\sigma^2dx^2$ are related by
   \begin{equation} \label{curvature}
    \kappa=\frac{\mathring\kappa}\sigma-\frac{\langle\mathring\nabla\sigma,
    \mathring\nu\rangle}{\sigma^2}
   \end{equation}
   where $\mathring\nu$ is the unit (for the Euclidean metric) normal 
   vector 
   to the curve, $\mathring\nabla$ is the Levi-Civit\`a connection of 
   the Euclidean metric, and $\langle\,,\,\rangle$ is the Euclidean inner product.
\end{lemma}
   
\begin{proof}
 This formula is well-known and follows from Lemma \ref{lemma.conformal2} and Remark \ref{rem.2for2}. 
   Let $\mathring\II^H$ be the second fundamental form of $H$ for the metric $dx^2$ and and unit normal $\mathring\nu$. Then $\frac1\sigma\mathring\nu$ has norm $1$ for the metric $\sigma^2dx^2$, so we can consider the second fundamental form $\II^H_\sigma$ of $H$ for the metric $\sigma^2dx^2$ and unit normal $\frac1\sigma\mathring\nu$.
    
     Let $s$ be a choice of arc-length coordinate on $H$ for the Euclidean metric $dx^2$.
    We have  
    $$\mathring\kappa=\mathring\II^{H}(\pa_s,\pa_s),$$
    and, since $\frac1\sigma\pa_s$ has norm $1$ for the metric $\sigma^2dx^2$,
    $$\kappa=\II^H_\sigma\left(\frac1\sigma\pa_s,\frac1\sigma\pa_s\right)
    =\frac1\sigma\mathring\II^H(\pa_s,\pa_s)-\frac{\langle\mathring\nabla\sigma,\mathring\nu\rangle}{\sigma^2}
      =\frac{\mathring\kappa(s)}\sigma-\frac{\langle\mathring\nabla\sigma,
    \mathring\nu\rangle}{\sigma^2}.$$
   This completes the proof.
\end{proof}

\begin{lemma} \label{lines}
   Consider the metric $\tilde g=\frac{dx^2}{\eta(|x|)^2}$ on 
   $\RR^2\setminus\{0\}$. 
   Then, for any $k\in\NN$ there exists a constant $C_{k}$ such that 
   the curvature $\tilde\kappa$ with respect $\tilde g$ of any Euclidean straight 
   line satisfies 
   $\left|\frac{d^k\tilde\kappa}{d\tilde s^k}\right|\leq C_{k}$,
   where $\tilde s$ is the arc-length for the metric $\tilde g$.
\end{lemma}

\begin{proof}
If $D$ is an oriented Euclidean straight line, possibly passing through $0$, 
then let $\tilde\kappa_D(x)$ be its curvature in the metric $\tilde g$ at a 
point $x\in D\setminus\{0\}$. Outside the Euclidean ball $\ball{0}{1/5}$ we have 
$\tilde g=\frac16dx^2$, so $\tilde\kappa_D(x)=0$ if $|x|\geq1/5$.

Let $k\in\NN$. We proceed by contradiction. Assume that there exists a sequence 
of Euclidean straight lines $(D_m)_{m\in\NN}$ and a sequence of points 
$(x_m)_{m\in\NN}$ with $x_m\in D_m\setminus\{0\}$ for all $m\in\NN$ and 
$\left|\frac{d^k\tilde\kappa_m}{d\tilde s^k}(x_m)\right|\to+\infty$ as 
$m\to+\infty$, where $\tilde\kappa_m:=\tilde\kappa_{D_m}$. Since rotations 
around $0$ are isometries of the metric $\tilde g$, we may assume that the 
Euclidean straight lines $D_m$, $m\in\NN$, are parallel.
 
Since $\frac{d^k\tilde\kappa_m}{ds^k}(x)=0$ if $|x|>1/5$, for $m$ large enough 
we have $|x_m|\leq1/5$. Hence, up to extracting a subsequence, we may assume 
that $x_m$ converges in the Euclidean distance to a limit point $x_{\infty}$ 
as $m\to+\infty$, because $\cup_{n \in \oN} \Omega_{n}$ is bounded,
see Assumption \ref{assumpt.fw}.
 
Assume that $x_\infty\neq0$. Then $x_m$ also converges to $x_\infty$ in the 
metric $\tilde g$ as $m\to+\infty$, and also $D_m$ converges smoothly on 
compact sets as $m\to+\infty$ to the Euclidean straight line $D_\infty$ passing 
through $x_\infty$ and parallel to the $D_m$, $m\in\NN$. Then, for suitable 
choices of orientations of the straight lines, 
$\frac{d^k\tilde\kappa_m}{d\tilde s^k}(x_m)\to
\frac{d^k\tilde\kappa_{D_\infty}}{d\tilde s^k}(x_\infty)$ as $m\to+\infty$, 
contradicting the hypothesis.
 
Consequently, $x_\infty=0$. In particular, for $m\in\NN$ large enough, 
$|x_m|<1/16$. Then, the homothety $h_m:=h_{0,1/(16|x_m|)}$ (See Notation 
\ref{not.homothety})
induces an isometry 
between $(B^*(0,2|x_m|),\tilde g)$ and $(B^*(0,1/8),\tilde g)$ (where $B^*$ 
is the ball minus the center), so $$\frac{d^k\tilde\kappa_m}{d\tilde s^k}(x_m)
=\frac{d^k\tilde\kappa_{h_m(D_m)}}{d\tilde s^k}(h_m(x_m)).$$ On the other hand, 
$|h_m(x_m)|=1/16$, so, up to extracting a subsequence, we may assume that 
$h_m(x_m)$ converges in the Euclidean distance to a limit point $y_{\infty}$ 
as $m\to+\infty$, with $|y_\infty|=1/16$. The convergence also holds in the 
metric $\tilde g$, and also $h_m(D_m)$ converges smoothly on compact sets as 
$m\to+\infty$ to the Euclidean straight line $D'_\infty$ passing through 
$y_\infty$ and parallel to the $D_m$, $m\in\NN$. Then, for suitable choices 
of orientations of the straight lines, $$\frac{d^k\tilde\kappa_m}{d\tilde s^k}(x_m)
=\frac{d^k\tilde\kappa_{h_m(D_m)}}{d\tilde s^k}(h_m(x_m))\to
\frac{d^k\tilde\kappa_{D'_\infty}}{d\tilde s^k}(y_\infty)$$ 
as $m\to+\infty$, contradicting the hypothesis.
\end{proof}

\begin{lemma} \label{lemma_curvature_boundary}
   For each $k\in\mathbb{N}$, $\frac{d^k\kappa}{ds^k}$ is 
   bounded on $\partial U$, where $s$ denotes arc-length 
   with respect to the metric $\hat g$ and $\kappa$ is the curvature of 
   $\pa U$.
\end{lemma}

\begin{proof}
The curvature of $\partial\Omega_{\infty}\setminus \maV_{\infty}$ for 
the metric $\hat g_{\infty}$ vanishes (indeed, around any point of one 
of these edges, the Euclidean reflexion across the edge is a local 
isometry of $\hat g_{\infty}$; alternatively, this also follows from 
Lemma \eqref{lemma.curvature}). Hence, it 
suffices to prove that, for each $k \in \NN$, there exists a 
constant $A_{k}>0$ such that 
$\left|\frac{d^k\kappa_{n}}{ds^k}\right|\leq A_{k}$ on $\partial\Omega_{n}$ 
for any $n\in\mathbb{N}^*$. Fix $j\in\{1,\dots,N\}$.

If $x\in \ball{p_{nj}}{\rho/(2n)}$, then $$\min_{p\in\mathcal{V}_{n}}|x-p|
=|x-p_{nj}|<\frac18,$$ so $r_{n}(x)=|x-p_{nj}|$. Consquently, the map 
$h_{p_j,n}$ induces an isometry between $(B^*(p_{nj},\rho/(2n)),\hat g_{n})$ 
and $(B^*(p_{j1},\rho/2), \hat g_1)$ (where $B^*$ is the ball minus the center). 
Since $c_{nj}\subset B^*(p_{nj},\rho/(2n))$ and $h_{p_j,n}(c_{nj})=c_{j1}$, 
we deduce that $\left|\frac{d^k\kappa}{ds^k}\right|$ is bounded 
on $\bigcup_{n=1}^{+\infty} \{n\}\times c_{nj}$ by 
$\max_{c_{j1}}\left|\frac{d^k\kappa_1}{ds^k}\right|$.

Also, since the balls $\ball{p_{1n}}{1/5},\dots,\ball{p_{Nn}}{1/5}$ are 
pairwise disjoint, each point of $(\RR^2\setminus\mathcal{V}_{n},
\hat g_{n})$ has 
a neighborhood which is isometric to an open set of 
$(\RR^2\setminus\{0\},\tilde g)$. Consequently, we 
deduce from Lemma \ref{lines} that 
$\left|\frac{d^k\kappa}{ds^k}\right|$ is bounded on 
$\bigcup_{n=1}^{+\infty} \{n\}\times\ell_{nj}$ by $C_{k}$.
Let 
$$A_{k}  \ede \max\left( \{C_{k}\}\cup\left\{\max_{c_{j1}}
\left| \frac{d^k\kappa_1}{ds^k} \right| 
\mid j\in\{1,\dots,N\}\right\} \right).$$
In conclusion, $\left|\frac{d^k\kappa}{ds^k}\right|$ is bounded on 
$\partial U$ by $A_{k}$.
\end{proof}

We are ready to state one of the main technical results.

\begin{proposition} \label{prop.q.Benoit}
   Let $\Omega_{\infty}$ be a straight polygonal domain
   and let $\Omega_{n}$ be as defined in \ref{ssec.Simon}.
   We let, as usual $U := \cup_{n \in \oN} \{n\} \times \Omega_{n}$ 
   and $r(n, x) := r_{n}(x)$. Then $(U, \pa U)$
   with the metric $\hat g = r^{-2} dx^2$ is a manifold of finite width 
   (that is, a manifold with boundary and 
   relative bounded geometry with $U \ni x \to \dist_{U}(x, \pa U)$
   bounded in $x$).
\end{proposition}

\begin{proof}
   Recall that $R = 1$. Let $(\Mce, \hat g)$ be as in 
   Proposition \ref{prop.Mce.bg}, in particular, 
   $\hat g = \frac{dx^2}{\eta(|x-\maV_{n}|)^2}$ on $\{n\} \times 
   \Omega_{n}$. We know from that proposition that $(\Mce, \hat g)$
   is a manifold with bounded geometry.

   The boundary of $U$ is $\pa U=\cup_{n = 0}^\infty \{n\} 
   \times \pa\Omega_{n}$. To prove that $U$ has bounded geometry, it is enough 
   to prove that $\pa U$ is a bounded geometry hypersurface in $\widehat M$. 
   Since $\pa U$ is one-dimensional, its second fundamental form $\II^{\pa U}$ 
   satisfies $\II^{\pa U}(\pa_s,\pa_s) = 
   - \hat g(\pa_s,\nabla_{\pa_s}\nu)=\kappa(s)$, 
   where $\nabla$ denotes the Riemannian connection of $\hat g$, $s$ arc-lenghth 
   on $\pa U$, $\nu$ the unit normal to $\pa U$ (with compatible orientation) 
   and $\kappa$ the geodesic curvature of $\pa U$ (see Remark \ref{rem.2for2}). 
   Consequently, the fact that 
   $\|(\nabla^{\pa U})^k \II^{\pa U} \|_{L^\infty} < \infty$ for all 
   $k \geq 0$ follows from Lemma \ref{lemma_curvature_boundary}. 
   This proves condition (iii) of 
   Definition \ref{def.hSurfaceBG}. (Condition (i) is obvious.)

   We claim that $r_{\pa U}>0$. We have 
   $r_{\pa U}=\inf_{n\in\NN\cup\{\infty\}}r_{\pa\Omega_{n}}$. 
   One has $r_{\pa\Omega_\infty}>0$ since, outside a compact region, 
   $\pa\Omega_\infty$ consists of parallel straight lines in each cylindrical 
   end $B_{1/8}(p_j)\setminus\{p_j\}$. Also, 
   by Lemma \ref{lemma.normalinjectivity}, there exists $C>0$ such that 
   $r_{\pa\Omega_{n}}\geq\min\left(C,\frac12\omega_{\pa\Omega_{n}}\right)$ 
   for all $n\in\NN$. So, it suffices to prove that 
   $\inf_{n\in\NN}\omega_{\pa\Omega_{n}}>0$.
   
   Assume that $\inf_{n\in\NN}\omega_{\pa\Omega_{n}}=0$. Then there exist a 
   function $\varphi:\NN\to\NN$, sequences $(x_{n})_{n\in\NN}$ and 
   $(t_{n})_{n\in\NN}$ such that $x_{n}\in\pa\Omega_{\varphi(n)}$, 
   $t_{n} \neq 0$, $x'_{n}:=\exp_{\varphi(n)}^\perp(x_{n},t_{n})\in\pa\Omega_{\varphi(n)}$,  
   $t_{n}\to0$ and $\varphi(n)\to\infty$ as $n\to\infty$, where 
   $\exp_{\varphi(n)}^\perp$ stands for $\exp^\perp$ in the metric 
   $r_{\varphi(n)}^{-2}dx^2$. We can assume that $|t_{n}|$ is minimal, 
   i.e., $\exp_{\varphi(n)}^\perp(x_{n},t)\notin\pa\Omega_{\varphi(n)}$ 
   if $|t|<|t_{n}|$. Extracting a subsequence, we can assume that $x_{n}$ 
   converges in the Euclidean distance to a point 
   $x_\infty\in\pa\Omega_\infty$, because $\cup_{n \in \oN} \Omega_{n}$ is bounded,
   see Assumption \ref{assumpt.fw}. Also, let $\gamma_{n}$ be the curve 
   parametrized by $t\mapsto\exp_{\varphi(n)}^\perp(x_{n},t)$ between 
   $x_{n}$ and $x'_{n}$. Since the functions $r_m$, $m\in\NN$, are uniformly 
   bounded, the Euclidean length of $\gamma_{n}$ also tends to $0$ as $n\to\infty$.
   
   Assume that $x_\infty\notin\maV_\infty$. The curve $\gamma_{n}$ is a geodesic 
   for the metric $r_{\varphi(n)}^{-2}dx^2$ meeting $\pa\Omega_{\varphi(n)}$ at 
   $x_{n}$ and $x'_{n}$, orthogonally at $x_{n}$. As $n\to\infty$, the Euclidean total 
   absolute curvature of $\gamma_{n}$ tends to $0$ (since its Euclidean length 
   tends to $0$ and the Euclidean curvature of geodesics for the metrics 
   $r_m^{-2}dx^2$, $m\in\NN$, is uniformly bounded in a neighborhood of $x_\infty$). 
   So, in the limit, we get two arcs of $\pa\Omega_\infty$ meeting at $x_\infty$. 
   This is impossible, since $\pa\Omega_\infty$ is embedded.
   
   Consequently, $x_\infty=p_j$ for some $j\in\{1,\dots,N\}$.  We consider the 
   curve $\Gamma$ defined as the union of $c_{j1}$ and the two half-lines starting 
   at the endpoints of $c_{j1}$ such that the union is smooth (in particular, these 
   half lines are parallel to the edges of $\pa\Omega_\infty$ that meet at $p_j$). 
   For the metric $\frac{dx^2}{|x-p_{j1}|^2}$, it has a positive normal injectivity 
   radius $r_\Gamma>0$ (its ends are asymptotic to half lines issued from $p_j$, 
   which are parallel geodesics).
   
   For $n$ large enough, one has 
   $\gamma_{n}\subset B_{1/16}(p_j)\subset B_{1/8}(p_{j\varphi(n)})$. But the 
   homothety $h_{p_j,\varphi(n)}$ maps isometrically 
   $\left(B_{1/8}(p_{j\varphi(n)}),r_{\varphi(n)}^{-2}dx^2\right)$ onto 
   $\left(B_{\varphi(n)/8}(p_{j1}),\frac{dx^2}{|x-p_{j1}|^2}\right)$ and 
   $\left(\pa\Omega_{\varphi(n)}\cap B_{1/8}(p_{j\varphi(n)}), 
   r_{\varphi(n)}^{-2}dx^2\right)$ onto $\left(\Gamma\cap B_{\varphi(n)/8}(p_{j1}),
   \frac{dx^2}{|x-p_{j1}|^2}\right)$. Consequently, $t_{n}\geq r_\Gamma>0$, 
   which contradicts the fact that $t_{n}\to0$ as $n\to\infty$. This concludes 
   the proof that $r_{\pa U}>0$.
   
   We now check that $U$ has finite width. By construction, all $\Omega_{n}$ 
   are included in $\overline\Omega_1$, which is compact. Let $n\in\NN$ and 
   $y\in\Omega_{n}$. Outside the balls $\ball{p_{nj}}{1/8}$, 
   $j\in\{1,\dots,N\}$, 
   we have $r_{n}\geq1/8$. So, if $y\notin \ball{p_{nj}}{1/8}$ for all 
   $j\in\{1,\dots,N\}$, then $\dist_{U}(y, \pa \Omega_{n})$ is less than a 
   constant that does not depend on $n$ and $y$. If $y\in \ball{p_{nj}}{1/8}$ 
   for some $j\in\{1,\dots,N\}$, then there is an arc of a Euclidean circle 
   centered at $p_{nj}$ that is contained in $\overline\Omega_{n}$ and that 
   joins $y$ to a point of $\pa\Omega_{n}$; since the length of this circle 
   for the metric $r_{n}^{-2}dx^2$ is $2\pi$, we get 
   $\dist_{U}(y, \pa \Omega_{n})\leq2\pi$. This concludes the proof.
\end{proof}

In particular, our domains $\Omega_{n}$ defined in this section 
and their ``limit'' $\Omega_{\infty}$ satisfy the conditions of 
Theorem \ref{theorem.main}, and hence we have the following result:

\begin{corollary}\label{cor.theorem.main}
   Let $\Omega_{n} \subset \RR^{2}$, $n \in \oN$, be the
   domains constructed in this section. Then there 
   exists $\delta_{U} > 0$ with the following property. 
   Let $|a| < \delta_{U}$, $a \le 1$, and $n \in \oN$. 
   Then there exists $C_H > 0$ such that, for any 
   $f_{n} \in L^2(\Omega_{n})$, $n \in \oN$,
   \begin{equation*}
      \|\Delta_{D, n}^{-1}f_{n}\|_{H^{1+a}} 
      \le C_H \|f_{n}\|_{L^2(\Omega_{n})}\,.
   \end{equation*}
\end{corollary}

As we have mentioned earlier, when the limit domain $\Omega_{\infty}$
is actually smooth (no corners), the Babu\v{s}ka-Kondratiev spaces 
we used become the usual Sobolev spaces. Even in that case, our 
results are new.

\def\cprime{$'$}

\end{document}